\pdfoutput=1
\documentclass[onefignum,onetabnum,pdftex]{siamart171218}
\usepackage{placeins}
\usepackage{color}
\usepackage[countmax]{subfloat}
\usepackage{float}
\usepackage{caption}
\usepackage{placeins} 
\usepackage{subfig}
\usepackage{standalone}
\usepackage{inputenc}
\graphicspath{{./}{./figures/}{./Figures/}}
\usepackage{tikz}
\usepackage{pgfplots}
 \pgfplotsset{compat=1.13}
\usepackage{pgfplotstable}
\usetikzlibrary{calc}
\definecolor{color2}{rgb}{0.872549019607843,0.019607843137255,0.0819607843137255}
\definecolor{dred}{rgb}{0.92,0,0}

\usepackage{bm} 
\newcommand\restr[2]{{
  \left.\kern-\nulldelimiterspace 
  #1 
  \vphantom{\big|} 
  \right|_{#2} 
  }} 
  \newcommand{\abs}[1]{\left\lvert#1\right\rvert}
\newcommand{\norm}[1]{\left\lVert#1\right\rVert}
\let\hat\widehat
\let\tilde\widetilde
\usepackage{lipsum}
\usepackage{amsfonts}
\usepackage{graphicx}
\usepackage{epstopdf}
\usepackage{algorithmic}
\ifpdf
  \DeclareGraphicsExtensions{.eps,.pdf,.png,.jpg}
\else
  \DeclareGraphicsExtensions{.eps}
\fi

\newsiamremark{remark}{Remark}
\newsiamremark{hypothesis}{Hypothesis}
\crefname{hypothesis}{Hypothesis}{Hypotheses}
\newsiamthm{claim}{Claim}

\headers{Minimal stabilization for IG methods on trimmed geometries}{A. Buffa, R. Puppi, and R. V\'azquez}

\ifpdf
\hypersetup{
  pdftitle={A minimal stabilization procedure for isogeometric methods on trimmed geometries},
  pdfauthor={A. Buffa, R. Puppi, and R. V\'azquez}
}
\fi

\begin{document}

\title{A minimal stabilization procedure for isogeometric methods on trimmed geometries\thanks{Submitted.
\funding{The authors were partially supported by ERC AdG project CHANGE n. 694515, by MIUR PRIN project "Metodologie innovative nella modellistica differenziale numerica", and by Istituto Nazionale di Alta Matematica (INdAM).}}}
\date{August 2018}

\author{A. Buffa  \footnotemark[3]\ \thanks{ Istituto di Matematica Applicata e
    Tecnologie Informatiche "Enrico Magenes" del CNR, Pavia, Italy.}
\and R. Puppi  \thanks{Chair of Numerical Modelling and Simulation, Institute of Mathematics, \'Ecole Polytechnique F\'ed\'erale de Lausanne, Lausanne, Switzerland (\email{annalisa.buffa@epfl.ch}, \email{riccardo.puppi@epfl.ch}, \email{rafael.vazquez@epfl.ch})}.
 \and R. V\'azquez  \footnotemark[3]\ \footnotemark[2] }
\maketitle

\begin{abstract}
Trimming is a common operation in CAD, and, in its simplest formulation, consists in removing superfluous parts from a geometric entity described via splines (a spline patch). After trimming the geometric description of the patch remains unchanged, but the underlying mesh is unfitted with the physical object. 
We discuss the main problems arising when solving elliptic PDEs on a trimmed domain. First we prove that, even when Dirichlet boundary conditions are weakly enforced using Nitsche's method, the resulting method suffers lack of stability. Then, we develop novel stabilization techniques based on a modification of the variational formulation, which allow us to recover well-posedness and guarantee accuracy. Optimal a priori error estimates are proven, and numerical examples confirming the theoretical results are provided.

 \end{abstract}



\section{Introduction}
Complex models are processed within Computer Aided Design (CAD) tools where several geometric manipulations are possible. Geometries are described as collection of their boundary surfaces, often defined as tensor-product splines or NURBS, and during the design process these surfaces can be joined, intersected or simply superflous parts can be cut away. All these Boolean operations act on the original surfaces through a common procedure of trimming.
When the superfluous surface areas are cut away, the visualization of the resulting surface changes, while its mathematical description does not. This description of the geometry is called ``boundary representation'' (B-rep) (see, e.g., \cite{piegl_ten_2005,Stroud2006}, or the recent review \cite{Marussig2017} and references therein) and is clearly not well suited for the simulation of PDEs.

Several efforts have been undertaken in the last years to improve the usability of CAD geometries in the solution of PDEs, especially thanks to the advent of Isogeometric Analysis (IGA) \cite{Cottrell:2009:IAT:1816404,HUGHES20054135} and its tremendous success (see, e.g., \cite{MR3761002,buffa,1078-0947_2019_1_241,CMAME-IGA2017,MR3811621}). The geometric modelling community has also provided important inputs to this scientific challenge \cite{Cohen15}, and, in this respect, volumetric representations (V-rep) are a major contribution \cite{Massarwi}. On the other hand, trimming remains a main tool for the design of complex models via Boolean operations, and basically all developments of IGA described above rely on strong requirements on the underlying geometric models and, in general, do not support trimmed geometric entities. The aim of this paper is to contribute to the design of isogeometric methods that robustly support trimming in the geometric description of the computational domain, and do not require the construction of a global re-parametrization (meshing).

Two main issues arising when dealing with trimmed geometries are the presence of elements unfitted with the boundary, making the research for efficient quadrature rules and the stable imposition of boundary conditions a challenge, and the existence of basis functions whose support has been cut, affecting the conditioning of the related linear system. 
Let us briefly review some of the most successful methods which have been proposed so far in this regard.

In connection with the construction of quadrature rules on trimmed elements let us mention, from the engineering side, the pioneering works \cite{NAGY2015165,Wang2015}, the shell analysis on geometric models with B-reps \cite{BREITENBERGER2015401,PHILIPP2016312}, and the finite cell method combined with IGA \cite{Dauge2015,RANK2012104,MR3003066,MR3706626}.

Concerning the lack of stability and the ill-conditioning of the stiffness matrix we should mention stabilization based on polynomial extrapolation in the parametric domain \cite{MARUSSIG201879,MARUSSIG2017497} and the so-called Cut-IGA method proposed in \cite{ELFVERSON20191,1812.08568}. The former traces its roots back to the groundbreaking work~\cite{doi:10.1137/S0036142900373208} where the authors, by modifying the discrete functional space, tackle both the stability and the conditioning issues. The theory was proved in \cite{doi:10.1137/S0036142900373208} for finite elements with (extended) B-splines, which corresponds to consider below the map $\F$ equal to the identity (i.e. without the isogeometric map), and the polynomial extrapolation may lead to sub-optimal convergence properties, as it will be discussed in Section~\ref{sec:construction}. The latter is a generalization of the Cut-FEM method \cite{BURMAN20101217,NME:NME4823,BURMAN2012328,Burman2012,Hansbo434217} born in the framework of fictitious domain methods for finite elements. It is based on a modification of the weak formulation of the problem by the addition a penalisation term on the boundary. This technique aims to deal at the same time with the ill-conditioning and the lack of stability of the bilinear form. Note that the penalisation term acts on the jumps of the normal derivatives of all orders over cut elements' boundaries, which can be very demanding if high-order B-splines are employed for the analysis (which is usually the case in IGA).

We now describe our simplified mathematical setting. Let $\Omega_0\subseteq \R^d$ (here $d=2,3$)  be a domain described by a bijective spline map $\mathbf{F}: (0,1)^d \to \Omega_0$, i.e. a patch in the isogeometric terminology, and let $\Omega_1,\ldots ,\Omega_N$ be bounded domains of $\R^d$. We assume that $\bigcup_{i=1}^N\overline{\Omega}_i$ are to be cut away from $\Omega_0$, i.e.
\begin{equation}\label{48}
\Omega=\Omega_0\setminus\bigcup_{i=1}^N\overline{\Omega}_i,
\end{equation}
and we assume that the computational domain $\Omega$ is Lipschitz.
After trimming the mathematical description of the domain remains unchanged, that is, the elements and basis functions fit the boundary of $\Omega_0$ instead of that of $\Omega$. In this paper, we focus on a simple Poisson problem, with weakly imposed boundary conditions, in the domain $\Omega$ described above. First of all, we discuss the difference between bad matrix conditioning and lack of stability. The former can be improved by modifying the chosen basis (preconditioning), while the latter needs to act on the bilinear form directly. 

Regarding the lack of stability, we propose a stabilization technique, inspired by \cite{Renard}, that acts only on those cut elements that affect stability. For example, in the case of a Neumann condition on the trimmed boundary, no stabilization is needed. The stabilization is ``minimal'' in the sense that no additional parameters are introduced, in contrast with the CutFEM~\cite{BURMAN2012328} and Finite Cell methods~\cite{Dauge2015} for instance.  Our stabilization is parameter free and its computation requires only local projections at the element level, and only for ``bad'' cut elements. We follow a ``local approach''~\cite{Marussig2017}, i.e. we modify the analysis, rather than the geometry, in order to be able to face the challenges arising from trimming. Moreover, we remain faithful to the so-called \emph{isogeometric paradigm}, in the sense that we just locally modify the weak formulation, while keeping the discrete functional space unaffected.



We present two different versions of the stabilization. The first one is based on polynomial extrapolation in the parametric domain, that is easier to implement from the numerical point of view, but suboptimal in some cases. The second one is a projection-based stabilization performed directly on the physical domain, which allows us to recover optimal a priori error estimates.

Concerning the conditioning issue, we do not have a sound solution to the problem, but we constructed tests to check the behaviour of the condition number of the stiffness matrix. Numerical evidences show that a rescaling of the stiffness matrix, coupled with our stabilization, greatly reduces the condition number, although it does not solve the conditioning issue in all configurations (see Section~\ref{subsection_ns_cond}). A clear theoretical understanding of this issue is beyond the scope of this paper and, in this regard, the interested reader is referred to \cite{MoBner} where a $L^p$-stable basis is constructed in the context of B-splines and to \cite{DEPRENTER2017297,Prenter:2020aa}.

The document is organized as follows. Section \ref{4} presents an overview of isogeometric analysis in trimmed domains. In Section \ref{35} we set the model problem, and explain the main challenges we need to face, namely integration, conditioning and numerical stability of the associated linear system. 
After having explained in detail in Section \ref{section_stability} the causes for the lack of stability of the simple Nitsche's formulation, in Section \ref{36} we present our new stabilization technique. Two possible constructions of the stabilization operator are suggested and analysed in Section~\ref{sec:construction}, and error estimates are provided in Section~\ref{sec:estimate}. Finally, we conclude by showing some numerical examples, obtained using the MATLAB library GeoPDEs \cite{VAZQUEZ2016523}, confirming the theoretical results.

\section{Isogeometric analysis on trimmed domains}\label{4}


\subsection{The univariate case}
For a more detailed introduction to isogeometric analysis, we refer the interested reader to the review article \cite{buffa}. Given two positive integers $p$ and $n$, we say that $\Xi:=\{\xi_1,\dots,\xi_{n+p+1} \}$ is a \emph{p-open knot vector} if
\begin{equation*}
\xi_1=\dots=\xi_{p+1}<\xi_{p+2}\le\dots \le \xi_n<\xi_{n+1}=\dots=\xi_{n+p+1}.
\end{equation*}
We assume $\xi_1=0$ and $\xi_{n+p+1}=1$. We also introduce $Z:=\{\zeta_1,\dots,\zeta_M \}$, the set of \emph{breakpoints}, or knots without repetitions, which forms a partition of the unit interval $(0,1)$. Note that
$
\Xi=\{\underbrace{\zeta_1,\dots,\zeta_1}_{m_1\;\text{times}},\underbrace{\zeta_2,\dots,\zeta_2}_{m_2\;\text{times}},\dots,\underbrace{\zeta_M,\dots,\zeta_M }_{m_M\;\text{times}}\},
$
where $\sum_{i=1}^Mm_i=n+p+1$. Moreover, we assume $m_j\le p$ for every internal knot and we denote $I_i:=(\zeta_i,\zeta_{i+1})$ and its measure $h_i:=\zeta_{i+1}-\zeta_{i}$, $i=1,\dots, M-1$.

We denote as $\hat{B}_{i,p}:[0,1]\to\R$ the i-th \emph{B-spline}, $1\le i\le n$, obtained using the Cox-de Boor formula, see for instance \cite{buffa}. Moreover, let $S_p(\Xi)=\operatorname{span}\{\hat{B}_{i,p}: 1\le i\le n \}
$ the vector space of univariate splines of degree $p$. $S_p(\Xi)$ can also be characterized as the space of piecewise polynomials of degree $p$ with $k_j:=p-m_j$ continuous derivatives at the breakpoints $\zeta_j$, $1\le j\le M$ (Curry-Schoenberg theorem).

Moreover, given an interval $I_j=\left( \zeta_j,\zeta_{j+1}\right)=(\xi_i,\xi_{i+1})$, we define its \emph{support extension} $\tilde I_j$ as
\begin{equation*}
\tilde I_j:=\operatorname{int}\bigcup\{\operatorname{supp}(\hat{B}_{k,p}): \operatorname{supp}(\hat{B}_{k,p})\cap I_j\ne\emptyset, 1\le k\le n \}=\left(\xi_{i-p},\xi_{i+p+1}\right).
\end{equation*}

\subsection{The multivariate case}
Let $d$ be the space dimension. Assume that $M_\ell,n_\ell\in\N$, $p\in\N$, $\Xi_\ell=\{\xi_{\ell,1},\dots,\xi_{\ell,n_\ell+p+1} \}$ and $Z_\ell=\{\zeta_{\ell,1},\dots,\zeta_{\ell,M_\ell} \}$ are given for every $1\le \ell\le d$. We set the degree $\mathbf{p}:=(p,\dots,p)$ and $\mathbf{\Xi}:=\Xi_1\times\dots\times\Xi_d$. With no risk of ambiguity we will write $p$ in place of $\mathbf p$. Note that the breakpoints of $Z_\ell$ form a Cartesian grid in the \emph{parametric domain} $\hat{\Omega}_0=(0,1)^d$. We define the \emph{parametric B\'ezier mesh}
\begin{equation*}
\hat{\mathcal{M}}_{0}=\{Q_{\mathbf{j}}=I_{1,j_1}\times\dots\times I_{d,j_d}: I_{\ell,j_\ell}=(\zeta_{\ell,j_\ell},\zeta_{\ell,j_\ell+1}): 1\le j_\ell\le M_\ell-1 \},
\end{equation*}
where each $Q_{\mathbf{j}}$ is called \emph{B\'ezier element}, with $h_{Q_{\mathbf{j}}}:=\operatorname{diam}\left( Q_{\mathbf{j}}\right)$. Let $h:=\max\{h_Q: Q\in\hat\mathcal M_0\}$, hence we denote $\hat{\mathcal{M}}_{0,h}=\hat{\mathcal{M}}_{0}$.

Throughout the manuscript we are going to rely on the following shape-regularity hypothesis, which allows us to assign $h_Q$ as the unique size of the element, without the necessity of dealing with the length of its edges separately. This hypothesis is implicitly used throughout the paper in any result involving the mesh size $h$.
\begin{assumption}\label{shape_regularity}
The family of meshes $\{\hat\mathcal M_{0,h}\}_h$ is assumed to be \emph{shape-regular}, that is, the ratio between the smallest edge of $Q\in\hat\mathcal M_{0,h}$ and its diameter $h_Q$ is uniformly bounded with respect to $Q$ and $h$.
\end{assumption}
\begin{remark}
The shape-regularity hypothesis implies that the mesh is \emph{locally-quasi uniform}, i.e. the ratio of the sizes of two neighboring elements is uniformly bounded (see \cite{BAZILEVS_1}).
\end{remark}

Let $\mathbf{I}:=\{\mathbf{i}=(i_1,\dots,i_d): 1\le i_\ell\le n_\ell \}$ be a set of multi-indices. For each $\mathbf{i}=(i_1,\dots,i_d)$, we define the set of \emph{multivariate B-splines}\\ $\{\hat{B}_{\mathbf{i},p}(\mathbf{\hat x})=\hat{B}_{i_1,p}(\hat x_1)\dots\hat{B}_{i_d,p}(\hat x_d): \mathbf{i}\in\mathbf{I}  \}.
$
Moreover, for an arbitrary B\'ezier element $Q_{\mathbf j}\in\hat {\mathcal M}_{0,h}$, we define its \emph{support extension} $\tilde Q_{\mathbf j}=\tilde I_{1,j_1}\times\dots\times\tilde I_{d,j_d}$,
where $\tilde I_{l,j_\ell}$ is the univariate support extension of the univariate case defined above. 

The \emph{multivariate spline space} in $\hat{\Omega}$ is defined as
$
S_{p}(\mathbf{\Xi})=\operatorname{span}\{\hat{B}_{\mathbf{i},p}(\mathbf{\hat x}):\mathbf{i}\in\mathbf{I} \},
$
which can also be seen as the space of piecewise multivariate polynomials of coordinate degree $p$ and with regularity across the B\'ezier elements given by the knots multiplicities. Note that $S_{p}(\mathbf{\Xi})=\bigotimes_{\ell=1}^d S_{p}(\Xi_\ell)$.
\begin{remark}
What has been said so far can be easily generalized to the case of Non-Uniform Rational B-Splines (NURBS) basis functions. See for instance \cite{Cottrell:2009:IAT:1816404}. 
\end{remark}
\begin{remark}
	The previous construction as well as what follows could be done in a more general setting considering different, but fixed, degrees in each Cartesian direction. In this case all the inequality constants appearing in the theoretical results would depend on the difference between the degrees, and would possibly explode if this difference is not kept bounded. Further generalizations to the anisotropic setting, either in terms of the mesh or in terms of the degree, would rely on the approximation theory in anisotropic Sobolev spaces (see, for instance, \cite{BEIRAODAVEIGA20121,10.2307/2156881}), These extensions are far from trivial, and out of the scope of this work.
	\end{remark}	

\subsection{Parametrization, mesh and approximation space for trimming domains}\label{47}
Let $\Omega_0\subset\R^d$ be the original domain before trimming. We assume that there exists a map $\mathbf{F}\in \left(S_{p^0}(\mathbf{\Xi}^0) \right)^d$ such that $\Omega_0=\mathbf{F}(\hat{\Omega}_0)$, for given degree $p^0$ and knot vector $\mathbf{\Xi}^0$. We define the (physical) \emph{B\'ezier mesh} as the image of the elements in $\hat{\mathcal{M}}_{0,h}$ through $\mathbf{F}$:
\begin{equation*}
\mathcal{M}_{0,h}:=\{K\subset\Omega: K=\mathbf{F}(Q), Q\in\hat{\mathcal{M}}_{0,h} \}.
\end{equation*}
We denote $h_K:=\operatorname{diam}\left(K\right)$ for each $K\in\mathcal M_{0,h}$.
To prevent the existence of singularities in the parametrization we make the following assumption.
\begin{assumption}\label{1}
The parametrization $\mathbf{F}:\hat{\Omega}_0\to\Omega_0$ is bi-Lipschitz. Moreover, $\restr{\mathbf{F}}{\overline{Q}}\in C^\infty(\overline{Q})$ for every $Q\in\hat{\mathcal{M}}_{0,h}$ and $\restr{\mathbf{F}^{-1}}{\overline{K}}\in C^\infty(\overline{K})$ for every $K\in\mathcal{M}_{0,h}$.
\end{assumption}
Some consequences of Assumption \ref{1} are the following.
\begin{enumerate}
\item $h_Q\approx h_K$, i.e. $\exists\ C_1>0, C_2>0$ such that $C_1 h_K\le h_Q \le C_2 h_K$;
\item $\exists\ C>0$ such that, $\forall\ Q\in\hat{ \mathcal{M}}_{0,h}$ such that $\mathbf{F}(Q)=K$, it holds\\ $\norm{D\mathbf{F}}_{L^{\infty}(Q)}\le C$ and $\norm{D\mathbf{F}^{-1}}_{L^{\infty}(K)}\le C$;
\item $\exists\ C_1>0, C_2>0$  such that  $\quad C_1\le\abs{\operatorname{det}(D\mathbf F(\hat {\bf x}))}\le C_2$,  for all $\hat {\bf x} \in \hat \Omega_0$.
\end{enumerate}
Let $\hat{V}_h=S_{p}(\mathbf{\Xi})$ be a refinement of $S_{p^0}(\mathbf{\Xi}^0)$ and define
\begin{equation*}
V_h=\operatorname{span}\{B_{\mathbf{i},p}(\mathbf{x}):=\hat{B}_{\mathbf{i},p}\circ\mathbf{F}^{-1}(\mathbf{x}):\mathbf{i}\in\mathbf{I} \},
\end{equation*}
where $\{\hat{B}_{\mathbf{i},p}: \mathbf{i}\in\mathbf{I} \}$ is a basis for $\hat{V}_h$.
Note that throughout this document $C$ will denote generic constants that may change at each occurrence, but that are always independent of the local mesh size.

Let us clarify the interpolation strategy we are going to rely upon in the rest of this manuscript. 
Given a function $u\in H^{s}(\Omega)$, $s\ge1$, we extend it using the Sobolev-Stein extension operator (see, for instance, Section 3.2 of~\cite{oswald}) $E:H^{s}(\Omega)\to H^{s}(\R^d)$ and denote its restriction to the un-trimmed domain as $\tilde u:=\restr{E(u)}{\Omega_0}$, for $u\in H^{s}(\Omega)$.
The spline quasi-interpolant operator (\cite{Buffa2016}) associated to the uncut mesh $\mathcal M_{0,h}$ is $\Pi_0: H^{s}(\Omega_0)\to V_h$. Hence, we are allowed to write $\norm{\Pi_0( u)}_{H^{s}(\Omega)}\le\norm{\Pi_0(\tilde u)}_{H^{s}(\Omega_0)}\le C \norm{\tilde u}_{H^{s}(\Omega_0)}\le C\norm{u}_{H^{s}(\Omega)}$ and, similarly, for every $0\le t\le s$, $\norm{u-\Pi_0( u)}_{H^t(\Omega)}\le \norm{\tilde u-\Pi_0( \tilde u)}_{H^t(\Omega_0)}\le C h^{s-t}\norm{\tilde u}_{H^s(\Omega_0)}\le C h^{s-t}\norm{u}_{H^s(\Omega)}$.

\section{The isogeometric formulation}\label{35}
At this point we suppose to trim $\Omega_0$ as explained in \eqref{48}, for simplicity, with $N=1$ , i.e. the new domain is $\Omega=\Omega_0\setminus\overline{\Omega}_1$. We denote the trimming curve as $\Gamma_{trim}=\partial\Omega\cap\partial\Omega_1$.
Let us consider the Poisson equation as model problem. Given $f\in L^2(\Omega)$, $g_D\in H^{\frac{1}{2}}(\Gamma_D)$ and $g_N\in H^{-\frac{1}{2}}(\Gamma_N)$, find $u:\Omega\to\R$ such that
\begin{equation}\label{3}
\begin{cases}
-\Delta u = f\qquad&\text{in}\;\Omega,\\
u=g_D\qquad&\text{on}\;\Gamma_D,\\
\displaystyle{\frac{\partial u}{\partial n}}=g_N\qquad&\text{on}\;\Gamma_N,
\end{cases}
\end{equation}
where $\Gamma_D\cup\Gamma_N=\Gamma=:\partial\Omega$ and $\mathring\Gamma_D\cap\mathring \Gamma_N=\emptyset$, and $\displaystyle{\frac{\partial u}{\partial n}}:=\nabla u\cdot \mathbf n$ is the normal derivative, with $\mathbf{n}$ the outward unit normal to $\Gamma$. Observe that, in general, $\Gamma_{trim}\cap\Gamma_{D}\ne\emptyset$.

Let us now develop and extend the notation introduced in Section \ref{4}, to adapt it to trimmed domains. 

The approximation space on the trimmed domain $\Omega$ is $\tilde{V}_h:=\operatorname{span}\{\restr{B_{\mathbf{i},\mathbf{p}}}{\Omega}:\mathbf{i}\in\mathbf{I} \}$. The new parametric B\'ezier mesh is
$
\hat{\mathcal{M}}_h=\{Q\in{\hat{\mathcal{M}}}_{0,h}: Q\cap \hat{\Omega}\ne\emptyset \},
$
where $\hat\Omega = \F^{-1}(\Omega)$. The physical mesh is
$
\mathcal{M}_h=\{\F(Q): Q\in\hat{\mathcal{M}}_h \}.
$
and the set of B\'ezier elements cut by the trimming curve is denoted as
$
\mathcal{G}_h=\{ K\in\mathcal{M}_h: \overline K\cap\Gamma_{trim}\ne\emptyset \}.
$

 For every $K\in \mathcal{M}_h$, let $h_K:=\operatorname{diam}(K)$, $h_{\max}:=\max_{K\in \mathcal{M}_h} h_K$ and $h_{\min}:=\min_{K\in \mathcal{M}_h} h_K$. We define $\mathsf{h}:\Omega\to\R^+$ to be the piecewise constant mesh-size function of $\mathcal M_h$ given by $\restr{\mathsf{h}}{K}:=h_K$. 

First of all, we make an assumption on how the mesh is cut by the boundary.
\begin{assumption}\label{mesh_assumptions}
There exists $C>0$ such that, $\forall\ h>0,\ \forall\ K\in\mathcal{M}_h$, it holds $\abs{\Gamma_K}\le C h_K^{d-1}$, where $\Gamma_K:=\Gamma_{D}\cap \overline K\ne\emptyset$.
\end{assumption}

We denote as $\hat\Gamma_{D} := \F^{-1}\left( \Gamma_{D}\right)$ and as $\hat {\mathbf{n}}$ its outward unit normal.

Since the point is to avoid a reparametrization and a remeshing of the trimmed domain, it is natural to see the analogy with fictitious domain methods, where the physical domain, with a possibly complicated topology, is immersed into a simpler, but unfitted, background mesh. Similarly to fictitious domain methods, we need to be able to impose essential boundary conditions when the mesh is not fitted with boundary. Following \cite{FERNANDEZMENDEZ20041257,STENBERG1995139}, we decide to employ Nitsche's method, which in its symmetric form reads as follows.

Find  $u_h\in \tilde{V}_h$ such that
\begin{equation}\label{12}
\begin{aligned}
\int_\Omega\nabla u_h\cdot \nabla v_h&-\int_{\Gamma_D}\frac{\partial u_h}{\partial n}v_h-\underbrace{\int_{\Gamma_D}u_h\frac{\partial v_h}{\partial n}}_\text{symmetry}+\underbrace{\beta\int_{\Gamma_D}\mathsf{h}^{-1}u_hv_h}_\text{stability}\\
=&\int_\Omega fv_h+\int_{\Gamma_N}g_Nv_h
\underbrace{-\int_{\Gamma_D}g_D\frac{\partial v_h}{\partial n}+\beta\int_{\Gamma_D}\mathsf{h}^{-1}g_D v_h}_\text{consistency},
\end{aligned}
\end{equation}
where $\beta>0$ is a penalization parameter. 

We define
\begin{equation*}
a_h(u_h,v_h):=\int_\Omega\nabla u_h\cdot \nabla v_h-\int_{\Gamma_D} \frac{\partial u_h}{\partial n}v_h-\int_{\Gamma_D} u_h \frac{\partial v_h}{\partial n}+\beta\int_{\Gamma_D}\mathsf{h}^{-1} u_h v_h.
\end{equation*}


Our main goal is to provide a minimal stabilization to make formulation \eqref{12} uniformly well-posed with respect to the mesh-size.

\section{Stability}\label{section_stability}
Firstly, we need to clarify what we actually mean by ``stability'' of the discrete variational problem \eqref{12}. We introduce the following mesh-dependent scalar product 
\begin{equation*}
\left( u_h,v_h \right)_{1,h,\Omega}:=\int_{\Omega} \nabla u_h \cdot\nabla v_h +\int_{\Gamma_D} \mathsf{h}^{-1}u_h v_h,
\end{equation*}
which induces the discrete norm
\begin{equation}\label{norm_nitsche}
\norm{u_h}^2_{1,h,\Omega}:=\norm{\nabla u_h}^2_{L^2(\Omega)}+\norm{\mathsf{h}^{-\frac{1}{2}}u_h}_{L^2(\Gamma_D)}^2.
\end{equation}
	\begin{definition}\label{definition_stability}
	Problem \eqref{12} is stable if there exist $\overline\beta>0$ and $\alpha>0$ such that for every $\beta\ge\overline\beta$, for every $h>0$ it holds that
\[
\alpha\norm{u_h}^2_{1,h,\Omega}\le a_h(u_h,u_h)\qquad\qquad\qquad\forall\ u_h\in\tilde V_h,
\]
and for every fixed $\beta \ge \overline{\beta}$ there exists $\gamma > 0$ such that for every $h > 0$ it holds that 
\[
a_h(u_h,v_h)\le \gamma \norm{u_h}_{1,h,\Omega}\norm{v_h}_{1,h,\Omega}\quad\ \forall\ u_h,v_h\in\tilde V_h.
\]
	\end{definition}
\begin{remark}
The main point of Definition~\ref{definition_stability} is to find $\overline\beta$, $\alpha$ and $\gamma$ that do not depend on the trimming configuration. In the following lines we will show with a numerical example that the formulation~\eqref{12} is not stable.	
\end{remark}
	\begin{remark}\label{remark_beta}
In Definition~\ref{definition_stability} we have followed \cite{STENBERG1995139}. Note that the coercivity constant only depends on $\overline\beta$, while the continuity constant depends on the penalization parameter $\beta$, and in particular it grows with $\beta$. This dependence of the constant on $\beta$ also occurs in Theorem~\ref{apriori_error_estimate} and in Proposition~\ref{prop7.4}. In practice, $\beta$ has to be chosen large enough (i.e. larger than $\overline\beta$), but as close as possible to $\overline \beta$, to avoid that the continuity constant deteriorates.
	\end{remark}
\begin{proposition}
	If for all $K\in\mathcal G_h$ we have $\Gamma_K=\emptyset$, then problem \eqref{12} is stable.
\end{proposition}
\begin{proof}
We refer the reader to~\cite{STENBERG1995139}.	
\end{proof}	
The following numerical experience shows that there exists a trimming configuration for which the formulation~\eqref{12} is not stable according to Definition~\ref{definition_stability}. In particular we show that for every fixed $\beta$ the continuity constant $\gamma$ may be arbitrarily large, for a given $h>0$. First, we notice that if $\gamma_\beta$ is the continuity constant corresponding to $\beta$, then $\gamma_\beta>\gamma_1$ for every $\beta>1$. So, we fix $\beta=1$ and show that $\gamma_1$ can be arbitrarily large.

Let us consider the following eigenvalue problem. Find $u_h\in\tilde{V}_h\setminus \{ 0\}$ and $\lambda_h\in \mathbb{R}$ such that
\begin{equation}
\int_\Omega\nabla u_h\cdot \nabla v_h-\int_{\Gamma_D} \frac{\partial u_h}{\partial n}v_h-\int_{\Gamma_D} u_h \frac{\partial v_h}{\partial n}+\int_{\Gamma_D}\mathsf{h}^{-1} u_h v_h= \lambda_h (u_h, v_h)_{1,h,\Omega}\quad\forall\ v_h\in \tilde{V}_h.\label{GEPbis}
\end{equation}
As the problem is symmetric, the continuity constant $\gamma_1$ equals the maximum eigenvalue of~\eqref{GEPbis}.
%
%
%
Let us consider $\Omega_0=(0,1)^2$ and as trimmed domain $\Omega = (0,1)\times(0,0.757)$. We fix $h=2^{-5}$ as mesh size and $p=3$ as degree. We construct a sequence of discrete spaces $\left(\tilde V_{h,\eps}\right)_\eps$ of degree $p$ and of class $C^2$ at the internal knots, starting from the uniform knot vectors $\Xi_x$, $\Xi_{y}$ and substituting in the latter the knot $0.75$ with $\overline\xi=0.757-\eps$, see Figure~\ref{domain_eps}. Basically, the horizontal knot line $\{(x,y):y=0.75\}$ is replaced by $\{(x,y):y=\overline\xi\}$, which is such that the smaller $\eps>0$ is, the closer to the trimming curve it becomes. 

In Figure \ref{fig_eigs_bad} we can see the dependence of the spectrum of \eqref{GEPbis} on the magnitude of $\varepsilon$. In particular, the magnitude of the largest generalized eigenvalue goes to infinity as $\varepsilon$ goes to $0$, implying that the discrete formulation \eqref{12} is not stable, as the continuity constant can be made arbitrarily large by reducing $\eps$.
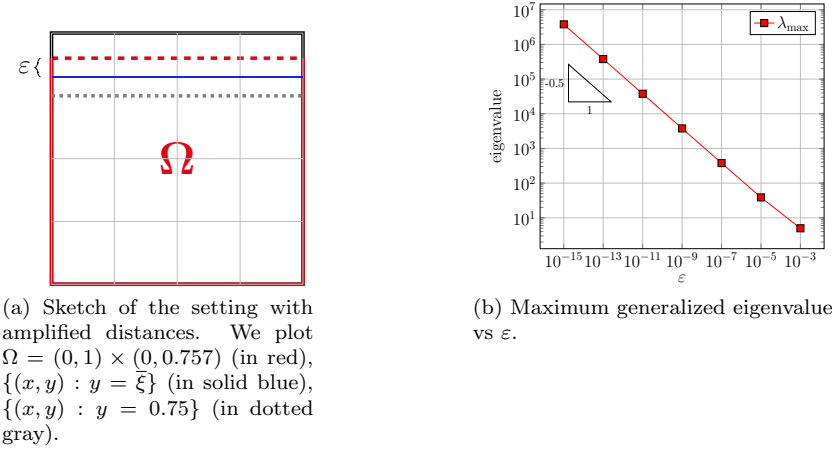
\begin{figure}[!htbp]
\centering
\subfloat[][Sketch of the setting with amplified distances. We plot $\Omega=(0,1)\times(0,0.757)$ (in red), $\{(x,y): y=\overline\xi\}$ (in solid blue), $\{(x,y): y=0.75\}$ (in dotted gray).\label{domain_eps}]
{  
  \begin{tikzpicture}
	
\definecolor{color2}{rgb}{0.872549019607843,0.019607843137255,0.0819607843137255}
\definecolor{blue}{rgb}{0,0,255}
\draw[dashed] (0,0) -- (1,0) ;
\draw[dashed] (0,0) -- (0,1) ;
\draw[color2, dashed, ultra thick] (0,3+0.6) -- (4,3+0.6);
\draw[ultra thick] (0,0) -- (4,0) -- (4,4) -- (0,4) -- (0,0);
\draw[color2,  ultra thick]  (0,3+0.6)--(0,0)--(4,0)--(4,3+0.6);
\coordinate (a) at (2,2);
\node at (a)   {\huge$\mathbf{\textcolor{color2}{\Omega}}$};
\begin{scope}
  \clip (0,0) -- (4,0) -- (4,4) -- (0,4) -- (0,0);
\end{scope}

\draw[step=2cm,lightgray,very thin] (0,0) grid (4,4);
\draw[lightgray,very thin] (0,1) -- (4,1);
\draw[lightgray,very thin] (1,0) -- (1,4);
\draw[lightgray,very thin] (3,0) -- (3,4);

   \draw[decoration={brace,raise=5pt},decorate]
  (0,3.3) -- node[left=6pt] {\large{$\varepsilon$}} (0,3.3+0.3);
  
  \draw[blue, thick](0,3.3) -- (4,3.3);
  
    \draw[gray, dotted, ultra thick](0,3) -- (4,3);
\end{tikzpicture}
  }
\hspace{2cm}
  \subfloat[][Maximum generalized eigenvalue vs $\varepsilon$.  \label{fig_eigs_bad}]
{
  \includestandalone[width=0.35\textwidth,keepaspectratio=true]{lambda_nostab}
  }
	\caption{Testing the lack of stability of formulation \eqref{40} with respect to trimming. }
\end{figure}
%
Going through the proof of stability, we clearly miss a discrete trace inequality which is uniform with respect to any mesh-trimming curve configuration, namely:
\begin{equation*}
\norm{\mathsf{h}^{\frac{1}{2}}\frac{\partial v_h}{\partial n}}_{L^2(\Gamma_K)}\le C \norm{\nabla v_h}_{L^2( K\cap\Omega)},
\end{equation*}
where $C$ does not depend neither on $K$ nor on $K\cap\Omega$.

At this point, in order to be able to deal with a well-posed, hence stable, problem we want to find a way to improve this discrete trace inequality, where the constant does not depend on how the trimmed boundary intersects the mesh.

\section{A new stabilization technique}\label{36}
The goal of this section is to present a new stabilization technique for the problem \eqref{12}. Our construction is inspired by the work of J. Haslinger and Y. Renard in \cite{Renard}.

Let us partition the elements of the B\'ezier mesh into two disjoint sub-families. 
\begin{definition}\label{def_goodandbad}
Let $\theta\in (0,1]$ and $Q\in\hat{\mathcal{M}}_h$. We say that $Q$ is a \emph{good element} if
\begin{equation}\label{good_bad_def}
\frac{\abs{\hat\Omega\cap Q}}{\abs{Q}}\ge\theta.
\end{equation}
Otherwise, $Q$ is a \emph{bad element}. Thanks to the regularity Assumption~\ref{1} on $\F$, this classification on the parametric elements induces naturally a classification on the physical elements. $\mathcal M_h^g$ stands for the collection of the good physical B\'ezier elements and $\mathcal M_h^b$ for the one of the bad physical elements. Note that $\mathcal M_h \setminus \mathcal G_h \subseteq \mathcal M_h^g$ and $\mathcal M_h^b\subseteq \mathcal G_h$. We denote the set of \emph{neighbors} of $K$ as:
	\begin{equation*}
	\mathcal N(K):=\{K'\in\mathcal M_h: \operatorname{dist}\left(K,K'\right)\le C \restr{\mathsf h}{K}\}\setminus\{K\},
	\end{equation*}
	where $C$ does not depend on the mesh size.
\end{definition}
The following assumption is not restrictive since it holds true if the mesh is sufficiently refined.
\begin{assumption}\label{assumption_neigh}
We assume that for any $K\in\mathcal M_h^b$, $\mathcal N (K) \cap \mathcal M _h^g\ne\emptyset$.
\end{assumption}


In what follows, we will use Assumption~\ref{assumption_neigh} to construct a stable representation of the normal flux of discrete functions. Let us assume that there exists an operator 
\begin{equation*}
R_h:\tilde{V}_h\to L^2(\Gamma_D)
\end{equation*}
which approximates the normal derivative on $\Gamma_D$ in a sense that will be specified. 
We propose the following stabilized formulation of problem \eqref{12}.

Find $u_h\in\tilde{V}_h$ such that
\begin{equation}\label{40}
\overline{a}_h(u_h,v_h)=\int_\Omega fv_h+\int_{\Gamma_N}g_Nv_h-\int_{\Gamma_D}g_DR_h(v_h)+\beta\int_{\Gamma_D}\mathsf{h}^{-1}g_D v_h\qquad\forall\; v_h\in \tilde{V}_h,
\end{equation}
where 
\begin{equation*}
\overline{a}_h(u_h,v_h):=\int_\Omega\nabla u_h\cdot \nabla v_h-\int_{\Gamma_D}R_h( u_h)v_h-\int_{\Gamma_D}u_h R_h(v_h)+\beta\int_{\Gamma_D}\mathsf{h}^{-1}u_hv_h.
\end{equation*}

\begin{theorem}\label{thm_well_posedness}
Suppose the following \emph{stability property} is satisfied: there exists a uniform $C>0$ such that for every $K\in\mathcal G_h$
\begin{equation}\label{41}
\norm{\mathsf{h}^{\frac{1}{2}}R_h(v_h)}_{L^2(\Gamma_K)}\le C\norm{\nabla v_h}_{L^2(\Omega\cap K')}\qquad\forall\; v_h\in \tilde{V}_h,
\end{equation}
where $K'=K$ if $K\in\mathcal M_h^g$, otherwise $K'\in\mathcal{N}(K)\cap \mathcal M_h^g$. Then problem \eqref{40} is stable in the sense of Definition~\ref{definition_stability} (modified accordingly).
\end{theorem}
\begin{proof}
For the continuity, let $u_h,v_h\in\tilde V_h$ and estimate
\begin{equation*}
\begin{aligned}
&\abs{\overline a_h(u_h,v_h)}\le  \norm{\nabla u_h}_{L^2(\Omega)}\norm{\nabla v_h}_{L^2(\Omega)}+\norm{\mathsf{h}^{\frac{1}{2}}R_h(u_h)}_{L^2(\Gamma_D)}\norm{\mathsf{h}^{-\frac{1}{2}}v_h}_{L^2(\Gamma_D)}\\
&+\norm{\mathsf{h}^{\frac{1}{2}}R_h(v_h)}_{L^2(\Gamma_D)}\norm{\mathsf{h}^{-\frac{1}{2}}u_h}_{L^2(\Gamma_D)} 
+ \beta \norm{\mathsf{h}^{-\frac{1}{2}}u_h}_{L^2(\Gamma_D)}\norm{\mathsf{h}^{-\frac{1}{2}}v_h}_{L^2(\Gamma_D)}\\
\le & \norm{u_h}_{1,h,\Omega}\norm{v_h}_{1,h,\Omega}+C\norm{\nabla u_h}_{L^2(\Omega)}\norm{v_h}_{1,h,\Omega}
+C\norm{\nabla v_h}_{L^2(\Omega)}\norm{u_h}_{1,h,\Omega}\\
&+\beta\norm{u_h}_{1,h,\Omega}\norm{v_h}_{1,h,\Omega}
\le C\norm{u_h}_{1,h,\Omega}\norm{v_h}_{1,h,\Omega},
\end{aligned}
\end{equation*}
where we employed first Cauchy-Schwarz inequality, the definition of the norm \eqref{norm_nitsche} and the stability property \eqref{41}. Take $u_h\in\tilde V_h$. Using Young inequality, with $\delta>0$, and, again, the stability property \eqref{41}, we obtain:
\begin{equation*}
\begin{aligned}
a_h(u_h,u_h)
\ge & \norm{\nabla u_h}_{L^2(\Omega)}^2 - \frac{1}{\delta}\norm{\mathsf{h}^{\frac{1}{2}}R_h(u_h)}^2_{L^2(\Gamma_D)}-\delta\norm{\mathsf{h}^{-\frac{1}{2}}u_h}^2_{L^2(\Gamma_D)}\\
 &+\beta \norm{\mathsf{h}^{-\frac{1}{2}}u_h}_{L^2(\Gamma_D)}^2
\ge  \left(1-\frac{C}{\delta}\right)\norm{\nabla u_h}_{L^2(\Omega)}^2+\left(\beta-\delta \right)\norm{\mathsf{h}^{-\frac{1}{2}}u_h}_{L^2(\Gamma_D)}^2,
\end{aligned}
\end{equation*}
from which we deduce the coercivity, provided $C<\delta<\beta$.
\end{proof}

\begin{remark}
In order for the solution of \eqref{40} to be a good approximation of $u$, it is clear that we will also need to quantify the error between $R_h(u_h)$ and $\frac{\partial u}{\partial n}$. This fact will be addressed in the next section.
\end{remark}

\section{Construction of the stabilization operator} \label{sec:construction}
The definition of the operator $R_h$ is not unique. As already observed, we seek for a stable approximation of the normal derivative on the trimmed part of the boundary, namely on $\Gamma_K$ for every $K\in\mathcal G_h$. Here, we propose two different constructions of such an operator.
\begin{itemize}
\item a \emph{stabilization in the parametric domain}: for each $K\in\mathcal M_h^b$ we take the (unique) polynomial extension of the pull-back of the functions of $\tilde V_h$ from $Q'=\F^{-1}(K')$ to $Q=\F^{-1}(K)$, where $K'$ is a good neighbor;
\item a \emph{stabilization in the physical domain}: for each $K\in\mathcal M_h^b$, we first $L^2$-project the spline functions restricted to the good neighbor $K'$ onto the polynomial space $\mathbb{Q}_p(K')$, then we take their (unique) polynomial extension up to $K$.
\end{itemize}

\begin{definition}[Stabilization in the parametric domain]\label{stabilization_parametric}
We define the operator $R_h$ locally as $\restr{R_h(v_h)}{K}:=R_K(v_h)$ $\;\forall\ K\in\mathcal{G}_h$, $\forall\ v_h\in\widetilde{V}_h$, where
\begin{itemize}
\item if $K\in\mathcal M_h^g$,
\begin{equation*}
 R_K(v_h):=\frac{\partial \restr{v_h}{K}}{\partial n};
 \end{equation*}
 \item if $K\in\mathcal M_h^b$, $K'\in\mathcal{N}(K)\cap \mathcal M_h^g$,
 \begin{equation*}
 R_K(v_h):=\frac{\partial\left( \mathcal{E}\left(\restr{\hat{v}_h}{Q'}\right)\circ \mathbf{F}^{-1}\right)}{\partial n},
 \end{equation*}
where $\mathcal{E}: \mathbb{Q}_p(Q')\to  \mathbb{Q}_p(Q'\cup Q)$ is the polynomial natural extension.
\end{itemize}
\end{definition}

\begin{definition}[Stabilization in the physical domain]\label{stabilization_physical}
An alternative stabilization operator can be defined by using the $L^2$-projection in the physical domain.
We define the operator $R_h$ locally as $\restr{R_h(v_h)}{K}:=R_K(v_h)$ $\forall\ K\in\mathcal{G}_h$, $\forall\ v_h\in\widetilde{V}_h$:
\begin{itemize}
\item if $K\in\mathcal M_h^g$,
\begin{equation*}
 R_K(v_h):=\frac{\partial \restr{v_h}{K}}{\partial n};
 \end{equation*}
 \item if $K=\mathbf{F}(Q)\in\mathcal M_h^b$, $K'\in\mathcal{N}(K)\cap \mathcal M_h^g$,
 \begin{equation*}
 R_K(v_h):=\frac{\partial\left( \mathcal{E}\left( P(\restr{v_h}{K'})\right)\right)}{\partial n},
 \end{equation*}
where $P: L^2\left(K'\right)\to\mathbb{Q}_p\left(K'\right)$ is the $L^2$-orthogonal projection
and\\ $\mathcal{E}: \mathbb{Q}_p(K')\to  \mathbb{Q}_p(K'\cup K)$ is the polynomial natural extension.
\end{itemize}
\end{definition}

\begin{remark}
Note that in the trivial case where $\F = \operatorname{\bf{Id}}$, the $L^2$-projection $P$, restricted to $\tilde V_h$, reduces to the identity operator and the two stabilizations coincide.
\end{remark}
\subsection{Properties of the stabilization in the parametric domain}
We are now up to verify if our choice of $R_h$ verifies the stability property \eqref{41}. Its proof relies on a series of quite technical results that are reported in the Appendix.

\begin{theorem}\label{theorem_stability}The stability property \eqref{41} holds for $R_h$ defined as in Definition \ref{stabilization_parametric}, i.e., there exists $C>0$ such that for every $K\in\mathcal G_h$
\begin{equation*}
\norm{\mathsf{h}^{\frac{1}{2}} R_h(v_h)}_{L^2(\Gamma_K)}\le C\norm{\nabla v_h}_{L^2(\Omega\cap K')}\qquad\forall\; v_h\in \tilde{V}_h,
\end{equation*}
where $K'=K$ if $K\in\mathcal M_h^g$, otherwise $K'\in\mathcal{N}(K)\cap \mathcal M_h^g$.
\end{theorem}
\begin{proof}
Fixed $K\in\mathcal G_h$, it is enough to prove
\begin{equation*}
\norm{\mathsf{h}^{\frac{1}{2}} \overline v_h}_{L^2(\Gamma_K)}\le C \norm{v_h}_{L^2(\Omega\cap K')},
\end{equation*}
for $v_h\in\tilde V_h$ such that $\restr{\overline v_h}{K}:=\mathcal{E}\left(\restr{\hat{v}_h}{Q'}\right)\circ \mathbf{F}^{-1}$, where $\mathcal E:\mathbb Q_p(Q')\to \mathbb Q_p(Q'\cup Q)$ and $K=\mathbf F(Q)$, $K'=\mathbf F (Q')\in\mathcal{N}(K)\cap \mathcal M_h^g$. We can restrict ourselves to the case $K\in\mathcal M_h^b$ with good neighbor $K'$. It holds:
\begin{equation}
\begin{aligned}
\norm{\overline v_h}^2_{L^2(\Gamma_K)}=&\int_{\Gamma_K}\abs{\overline v_h}^2\mathrm{d}S=\int_{\mathbf{F}^{-1}(\Gamma_K)}\abs{\hat{v}_h}^2\abs{\operatorname{det}\left(D\mathbf{F}\right)}\norm{D\mathbf{F}^{-1}\bf{\hat{n}}}\mathrm{d}\hat{S}\\
\le &C \int_{\mathbf{F}^{-1}\left(\Gamma_K \right)}\abs{\hat  v_h}^2\mathrm{d}\hat S= C\norm{\hat v_h}^2_{L^2(\hat\Gamma_D\cap Q)},\label{-1}
\end{aligned}
\end{equation}
where  we have used $\mathbf{F}^{-1}(\Gamma_K)=\mathbf{F}^{-1}(\Gamma_D)\cap \mathbf{F}^{-1}(K)= \hat\Gamma_D\cap Q$, because $\mathbf{F}$ preserves boundaries (as homeomorphisms do). We then have:
\begin{equation*}
\begin{aligned}
\norm{\hat v_h}_{L^2(\hat\Gamma_D\cap Q)} \le  \abs{\hat\Gamma_D\cap Q}^{\frac{1}{2}}\norm{\hat v_h}_{L^\infty(\hat\Gamma_D\cap Q)}
\le \abs{\hat\Gamma_D\cap Q}^{\frac{1}{2}}\norm{\hat v_h}_{L^\infty(Q)}.
\end{aligned}
\end{equation*}
In the first inequality we have used H\"{o}lder inequality. Now, we employ Lemma~\ref{51} and  Assumption~\ref{mesh_assumptions}:
\begin{equation*}
\norm{\hat v_h}^2_{L^2(\hat\Gamma_D\cap Q)} \le C\abs{\hat\Gamma_D\cap Q}^{\frac{1}{2}}\norm{\hat v_h}_{L^\infty(Q')} \le C h_{}^{\frac{d-1}{2}}\norm{\hat v_h}_{L^\infty(Q')}.
\end{equation*}

At this point, notice that we can use Lemma~\ref{50} because $\frac{\abs{\Omega\cap K'}}{\abs{K'}}\ge\theta$ implies $\frac{\abs{\hat\Omega\cap Q'}}{\abs{Q'}}\ge C\theta_{\min}$, where $C$ depends just on $\mathbf{F}$, thanks to Assumption \ref{1}.


Let us continue with the inequalities:
\begin{equation}
\norm{\hat v_h}_{L^2(\hat\Gamma_D\cap Q)} \le  C h_{}^{-\frac{1}{2}} \norm{\hat v_h}_{L^2(\hat\Omega\cap Q')} \le C h_{}^{-\frac{1}{2}}\norm{v_h}_{L^2(\Omega\cap K')}. \label{-2}
 \end{equation}
Gathering together \eqref{-1} and \eqref{-2}, we conclude the proof.
\end{proof}

In what follows, we analyse the approximation properties of the operator $R_h$, and provide estimates that will be used in Section \ref{sec:estimate} to deduce a complete error estimate.
%
\begin{proposition}\label{thmchenonviene}
Let $\frac{1}{2}< k\le p$. There exists $C>0$ such that for every $K\in\mathcal G_h$
\begin{itemize}
\item if $\frac{1}{2}<k < p-\frac{1}{2}$, for every $v\in H^{k+1}(\Omega)$,:
\begin{equation*}
\begin{aligned}
\norm{\mathsf{h}^{\frac{1}{2}}\left(R_h\left(\Pi_0(\tilde v)\right)-\frac{\partial \tilde v}{\partial n}\right)}_{L^2(\Gamma_K)}\le C  h_{\max}^{k}\norm{\tilde v}_{H^{k+1} \left(\tilde K\cup \tilde{K}'\right)},
\end{aligned}
\end{equation*}
where $K'=K$ if $K\in\mathcal M_h^g$, otherwise $K'\in\mathcal{N}(K)\cap \mathcal M_h^g$;
\item if $p-\frac{1}{2}\le k\le p$ and each internal knot line is not repeated, for every $v\in H^{k+1}(\Omega)$, for all $\eps>0$,
\begin{equation*}
\begin{aligned}
\norm{\mathsf{h}^{\frac{1}{2}}\left(R_h\left(\Pi_0( \tilde v)\right)-\frac{\partial \tilde v}{\partial n}\right)}_{L^2(\Gamma_K)}\le C h_{\max}^{p-\frac{1}{2}-\eps}\norm{\tilde v}_{H^{k+1} \left(\tilde K\cup \tilde{K}'\right)},
\end{aligned}
\end{equation*}
where $K'=K$ if $K\in\mathcal M_h^g$, otherwise $K'\in\mathcal{N}(K)\cap \mathcal{M}_h^g$.
\end{itemize}
\end{proposition}

\begin{proof}
First of all, let $v\in H^{k+1}(\Omega)$, with $\frac{1}{2}< k\le p$. We take $K\in\mathcal G_h$. Let us distinguish two cases: either $K\in\mathcal M_h^g$ or $K\in\mathcal M_h^b$. 

If $K\in\mathcal M^g$. We use Lemma~\ref{lemma_hansbo} and standard approximation results:
\begin{equation*}
\begin{aligned}
&\norm{\mathsf{h}^{\frac{1}{2}}\left(R_h\left({\Pi_0}(\tilde v)\right)-\frac{\partial \tilde v}{\partial n}\right)}_{L^2(\Gamma_K)}^2=\norm{\mathsf{h}^{\frac{1}{2}}\left(\frac{\partial{\Pi_0}(\tilde v)}{\partial n}-\frac{\partial \tilde v}{\partial n}\right)}^2_{L^2(\Gamma_K)}\\
\le&  C\Big( \norm{\nabla{\Pi_0}(\tilde v)-\nabla \tilde v}^2_{L^2(K)}
+\norm{\mathsf{h}(\nabla{\Pi_0}( \tilde v)-\nabla  \tilde v)}^2_{H^1(K)} \Big)\\
\le& C \left(\norm{\mathsf{h} ^{k}\tilde v}^2_{H^{k+1}(\tilde K)}+\norm{\mathsf{ h}^{k}\tilde v}^2_{H^{k+1}(\tilde K)} \right)
\le 2Ch_{\max}^{k}\norm{\tilde v}^2_{H^{k+1}(\tilde K)}.
\end{aligned}
\end{equation*}

If $K=\mathbf F(Q)\in \mathcal M_h^b$ and $K'=\mathbf F(Q')\in\mathcal{N}(K)\cap \mathcal M_h^g$ be its good neighbor. We easily obtain:
\begin{equation}
\begin{aligned}
&\norm{\mathsf{h}^{\frac{1}{2}}\left(R_h\left({\Pi_0}(\tilde v)\right)-\frac{\partial \tilde v}{\partial n}\right)}_{L^2(\Gamma_K)}
\le C\norm{\mathsf{h}^{\frac{1}{2}}\left(\frac{\partial}{\partial n} \mathcal{E}\left(\restr{\Pi_0\left(\tilde v\right) \circ\F}{Q'} \right)-\frac{\partial\hat{ \tilde v}}{\partial n}\right)}_{L^2(\hat\Gamma_D\cap Q)}\\
\le& C\Big( \norm{\mathsf{h}^{\frac{1}{2}}\frac{\partial}{\partial n} \left(\mathcal{E}\left(\restr{\Pi_0\left(\tilde v\right) \circ\F}{Q'}\right) -\Pi_0\left( \tilde v\right)\circ\F\right)}_{L^2(\hat\Gamma_D\cap Q)}
\\&+ \norm{\mathsf{h}^{\frac{1}{2}}\frac{\partial}{\partial n}\left(\Pi_0\left(\tilde v\right)\circ\F-\hat{\tilde v}\right)}_{L^2(\hat\Gamma_D\cap Q)}\Big)\label{primo}.
\end{aligned}
\end{equation}
The second term converges as expected because of the properties of spline quasi-interpolants \cite{Buffa2016}. We focus on the first one. Let $\hat q = q\circ \F \in\mathbb{Q}_p(\mathbb{R}^d)$ be a global polynomial. Note that, trivially, $\mathcal{E}(\restr{\hat q}{Q'})=\restr{\hat q}{Q'}$. By triangular inequality:
\begin{equation}
\begin{aligned}
&\Big\| \mathsf{h}^{\frac{1}{2}}\frac{\partial}{\partial n} \left(\mathcal{E}\left(\restr{\Pi_0\left(\tilde v\right) \circ\F}{Q'}\right) -\Pi_0\left(\tilde  v\right)\circ\F\right)\Big\|_{L^2(\hat\Gamma_D\cap Q)}\\
\le&\norm{\mathsf{h}^{\frac{1}{2}}\frac{\partial}{\partial n}\mathcal{E}\left(\restr{\Pi_0\left( \tilde v\right) \circ\F}{Q'}- \hat q\right)}_{L^2(\hat\Gamma_D\cap Q)}+\norm{\mathsf{h}^{\frac{1}{2}}\frac{\partial}{\partial n}\left(\hat q-\Pi_0\left(\tilde v\right)\circ\F \right)}_{L^2(\hat\Gamma_D\cap Q)}.\label{eq2}
\end{aligned}
\end{equation}
Using Corollary~\ref{corollary_trace}, we can bound the last term of \eqref{eq2} as follows:
\begin{equation}
\norm{\mathsf{h}^{\frac{1}{2}}\frac{\partial}{\partial n}\left(\Pi_0\left(\tilde v\right)\circ\F-\hat q\right)}_{L^2(\hat\Gamma_D\cap Q))}\le C\norm{\left(\Pi_0\left(\tilde v\right)\circ\F-\hat q\right)}_{H^1(Q)}\label{secondo}.
\end{equation}
The first term of \eqref{eq2} can be bounded using the stability property of $R_h$, given in Theorem \ref{theorem_stability}:
\begin{equation}
\norm{\mathsf{h}^{\frac{1}{2}}\frac{\partial}{\partial n}\mathcal{E}\left(\restr{\Pi_0\left(\tilde v\right) \circ\F}{Q'}- \hat q\right)}_{L^2(\hat\Gamma_D\cap Q)}\le C\norm{\nabla\left(\Pi_0\left( \tilde v\right) \circ\F- \hat q\right)}_{L^2(\hat\Omega\cap Q')}\label{terzo}.
\end{equation}
Thus, combining \eqref{eq2}, \eqref{secondo} and \eqref{terzo}, we obtain:
\begin{equation}
\label{terzobis}
\begin{aligned}
&\Big\| \mathsf{h}^{\frac{1}{2}}\frac{\partial}{\partial n} \left(\mathcal{E}\left(\restr{\Pi_0\left(\tilde v\right) \circ\F}{Q'}\right) -\Pi_0\left(\tilde  v\right)\circ\F\right)\Big\|_{L^2(\hat\Gamma_D\cap Q)}\\
\le &C\norm{\left(\restr{\Pi_0\left(\tilde v\right) \circ\F}{Q'}-\hat q\right)}_{H^1(Q\cup Q')}\\
\le& C \Big(\norm{\left(\Pi_0\left(\tilde v\right)- \tilde v\right)\circ\F}_{H^1(Q\cup Q')}+\norm{\left(\tilde v\circ\F-\hat q\right)}_{H^1(Q\cup Q')}  \Big).
\end{aligned}
\end{equation}

Again, the first term converges as expected by standard approximation results. Concerning the other term, there are some issues, related to the regularity of the parametrization. By the theory of \emph{bent Sobolev spaces} (see \cite{BAZILEVS_1}), we have $\tilde v\in H^{k+1}(\Omega_0)$, but, in general, $\restr{\tilde v\circ\mathbf F}{Q\cup Q'}\notin H^{k+1}(Q\cup Q')$, since it is bent by $\F$, a spline of degree $p$ and regularity $p-1$ (under the assumption that internal knot lines are not repeated). It holds, indeed, that  $\restr{\tilde v\circ\mathbf F}{Q\cup Q'}\in H^{r+1}(Q\cup Q')$, where $r+1:=\min \{k+1,p+\frac{1}{2}-\eps \}$, hence $0\le r\le k$ and $0\le r\le p-\frac{1}{2}-\eps$. So, the following inequality follows:
\begin{equation*}
\begin{aligned}
\norm{ \tilde v\circ\F-\hat q}_{H^1(Q\cup Q')}\le C h_{\max}^r\norm{\tilde  v\circ\F}_{H^{r+1}(\tilde Q\cup \tilde Q')},
\end{aligned}
\end{equation*}
where $0\le r\le k$ and $0\le r\le p-\frac{1}{2}-\eps$, for any $\eps>0$. Hence, pushing forward to the physical domain:
\begin{equation}
\begin{aligned}
\norm{\tilde v- q}_{H^1(K\cup K')}\le C h_{\max}^r\norm{\tilde v}_{H^{r+1}(\tilde K\cup \tilde K')}.\label{quarto}
\end{aligned}
\end{equation}
Hence, from \eqref{terzobis} and \eqref{quarto}, we deduce:
\begin{equation}
\begin{aligned}\label{ineq_incasinata}
&\norm{\mathsf{h}^{\frac{1}{2}}\left(R_h\left({\Pi_0}(\tilde v)\right)-\frac{\partial \tilde v}{\partial n}\right)}_{L^2(\Gamma_K)}\\
\le& C\Big( h_{\max}^{k}\norm{\tilde v}_{H^{k+1}\left(\tilde K\cup \tilde K'\right)}
+ h_{\max}^{r}\norm{\tilde v}_{H^{r+1}\left(\tilde K\cup \tilde K'\right)}\Big).
\end{aligned}
\end{equation}
We want to rewrite inequality \eqref{ineq_incasinata} by distinguishing two cases.
\begin{itemize}
\item $\frac{1}{2}<k<p-\frac{1}{2}$. In this case, 
\begin{equation}
\norm{\mathsf{h}^{\frac{1}{2}}\left(R_h\left({\Pi_0}(\tilde v)\right)-\frac{\partial \tilde v}{\partial n}\right)}_{L^2(\Gamma_K)}\le Ch_{\max}^{r}\norm{\tilde v}_{H^{k+1}\left(\tilde K\cup \tilde K'\right)} ,
\end{equation}
for any $0\le r \le k$. Hence,
\begin{equation}
\norm{\mathsf{h}^{\frac{1}{2}}\left(R_h\left({\Pi_0}(\tilde v)\right)-\frac{\partial \tilde v}{\partial n}\right)}_{L^2(\Gamma_K)}\le Ch_{\max}^{k}\norm{\tilde v}_{H^{k+1}\left(\tilde K\cup \tilde K'\right)} .
\end{equation}
\item  If $p-\frac{1}{2}\le k\le p$, then
\begin{equation}
\norm{\mathsf{h}^{\frac{1}{2}}\left(R_h\left({\Pi_0}(\tilde v)\right)-\frac{\partial \tilde v}{\partial n}\right)}_{L^2(\Gamma_K)}\le  Ch_{\max}^{p-\frac{1}{2}-\eps}\norm{\tilde v}_{H^{k+1}\left(\tilde K\cup \tilde K'\right)},
\end{equation}
for any $\eps>0$.
\end{itemize}
\end{proof}

\begin{remark}\label{remark1}
Note that if $\frac{1}{2}< k<p-\frac{1}{2}$, the estimate is optimal. In the case $p-\frac{1}{2}\le k\le p$ the estimate is sub-optimal, instead. As already mentioned during the proof, this is due to the fact that $u\in H^{k+1}(K\cup K')$ does not imply $u\circ\mathbf F\in H^{k+1}(Q\cup Q')$: if the knot line between $K$ and $K'$ is not repeated, namely $\F \in C^{p-1}(Q\cup Q')$, then it holds $u\circ\mathbf F\in H^{r+1}(Q\cup Q')$ with $r+1:=\min\{k+1,p+\frac{1}{2}-\eps\}$. Moreover, if the parametrization is less regular than requested in the hypotheses of Proposition~\ref{thmchenonviene}, then the sub-optimality may be even worse. More precisely, if $\F\in C^s\left(Q\cup Q'\right)$, which is the case if the knot line is repeated $p-s$ times, then we have $r+1:=\min\{k+1,s+\frac{3}{2}-\eps\}$. We will see an example of this sub-optimal behaviour in the worst case scenario of $s=0$ in Section~\ref{sec:apriori}.
\end{remark}
\begin{remark}\label{remark6.7}
Any method based on polynomial extrapolation of the B-splines in the parametric domain may also suffer of this sub-optimality depending on the regularity of the isogeometric map $\mathbf F$, because the theory of bent Sobolev spaces from~\cite{BAZILEVS_1} cannot be applied. In particular the method of extended B-splines which works very well in the parametric domain~\cite{doi:10.1137/S0036142900373208}, may suffer a lack of accuracy in the isogeometric setting~\cite{MARUSSIG201879,MARUSSIG2017497}.
\end{remark}
\subsection{Properties of the stabilization in the physical domain}
\begin{theorem}\label{thm1234}
The stability property \eqref{41} holds for $R_h$ defined as in Definition \ref{stabilization_physical}, i.e., there exists $C>0$ such that for every $K\in\mathcal G_h$
\begin{equation*}
\norm{\mathsf{h}^{\frac{1}{2}} R_h(v_h)}_{L^2(\Gamma_K)}\le C\norm{\nabla v_h}_{L^2(\Omega\cap K')}\qquad\forall\; v_h\in \tilde{V}_h,
\end{equation*}
where $K'=K$ if $K\in\mathcal M^g$, otherwise $K'\in\mathcal{N}(K)\cap \mathcal{M}^g$.
\end{theorem}
\begin{proof}
Let us start applying H\"{o}lder inequality and Lemma~\ref{51}:
\begin{equation*}
\begin{aligned}
\norm{R_h(v_h)}_{L^2(\Gamma_ K)}=&\norm{\frac{\partial}{\partial n}\mathcal{E}\left(P(\restr{v_h}{K'})\right)}_{L^2(\Gamma_K)} 
\le \sqrt{\abs{\Gamma_K}}\norm{\frac{\partial}{\partial n}\mathcal{E}\left(P(\restr{v_h}{K'})\right)}_{L^\infty(\Gamma_K)}\\
\le & \sqrt{\abs{\Gamma_K}}\norm{\nabla \mathcal{E}\left(P(\restr{v_h}{K'})\right)}_{L^\infty(K)}
\le C \sqrt{\abs{\Gamma_K}}\norm{\nabla P(\restr{v_h}{K'})}_{L^\infty(K')}.
\end{aligned}
\end{equation*}
We finish with Lemma~\ref{50}, Assumption \ref{mesh_assumptions} and the stability of the $L^2$-orthogonal projection $P$, see for instance~\cite{bramble},
\begin{equation*}
\begin{aligned}
\norm{\mathsf{h}^{\frac{1}{2}}R_h(v_h)}_{L^2(\Gamma_ K)}\le& C h_{}^{-\frac{d}{2}}\sqrt{\abs{\Gamma_K}}\norm{\mathsf{h}^{\frac{1}{2}}\nabla P(\restr{v_h}{K'})}_{L^2(\Omega\cap K')}\\
\le& C \norm{\nabla P(v_h)}_{L^2(\Omega\cap K')}
\le C \norm{\nabla v_h}_{L^2(\Omega\cap K')}.
\end{aligned}
\end{equation*}
We conclude by summing over $K\in\mathcal G_h$.
\end{proof}

\begin{proposition}\label{approx_physical}
Let $\frac{1}{2}< k\le p$. There exists $C>0$ such that for every $K\in\mathcal G_h$
\begin{equation*}
\norm{\mathsf{h}^{\frac{1}{2}}\left(R_h\left(\Pi_0(\tilde v)\right)-\frac{\partial \tilde v}{\partial n}\right)}_{L^2(\Gamma_K)}\le C h_{}^{k}\norm{\tilde v}_{H^{k+1}\left(\tilde K\cup \tilde{ K}'\right)}\qquad\forall\  v\in H^{k+1}(\Omega),
\end{equation*}
where $K'=K$ if $K\in\mathcal M_h^g$, otherwise $K'\in\mathcal{N}(K)\cap \mathcal M_h^g$.
\end{proposition}

\begin{proof}
We can focus on the case $K\in\mathcal M_h^b$. Let $K'\in\mathcal{N}(K)\cap \mathcal M_h^g$ and $q\in\mathbb{Q}_p\left( K'\right)$.
\begin{equation}
\begin{aligned}
&\norm{\mathsf{h}^{\frac{1}{2}}\left(R_h\left(\Pi_0\left(\tilde v\right)\right)-\frac{\partial\tilde v}{\partial n}\right)}_{L^2(\Gamma_K)}=\norm{\mathsf{h}^{\frac{1}{2}}\frac{\partial}{\partial n}\left(\mathcal{E}\left(P\left(\restr{\Pi_0\left(\tilde v\right)}{K'}\right)\right)-\tilde v\right)}_{L^2(\Gamma_K)}\\
\le & \norm{\mathsf{h}^{\frac{1}{2}}\frac{\partial}{\partial n}\left(\mathcal{E}\left(P\left(\restr{\Pi_0\left(\tilde v\right)}{K'}\right)\right)-q\right)}_{L^2(\Gamma_K)}+\norm{\mathsf{h}^{\frac{1}{2}}\frac{\partial}{\partial n}\left(q-\tilde v\right)}_{L^2(\Gamma_K)}.\label{passage_a}
\end{aligned}
\end{equation}
Let us focus on the first term. After having observed that $P(q)=q$, we apply the stability property proved in Theorem \ref{thm1234} and, again, triangular inequality
\begin{equation}
\begin{aligned}
&\norm{\mathsf{h}^{\frac{1}{2}}\frac{\partial}{\partial n}\left(\mathcal{E}\left(P\left(\restr{\Pi_0\left(\tilde v\right)}{K'}\right)\right)-q\right)}_{L^2(\Gamma_K)}\le  C \norm{\nabla\left( \Pi_0\left(\tilde v\right)-q\right)}_{L^2(\Omega\cap K')} \\
 \le &  C \Big(\norm{\nabla\left(q-\tilde v\right)}_{L^2(\Omega\cap K')}
 +\norm{\nabla\left(\tilde v-\Pi_0\left(\tilde v\right)\right)}_{L^2(\Omega\cap K')} \Big)\label{passage_b}
\end{aligned}
\end{equation}
We choose $q=P\left( v\right)$. Note that the second term of \eqref{passage_a} converges as expected by the approximation properties of the $L^2$-projection. Plugging \eqref{passage_b} into \eqref{passage_a}:
\begin{equation}
\begin{aligned}
\norm{\mathsf{h}^{\frac{1}{2}}\left(R_h\left(\Pi_0\left(\tilde v\right)\right)-\frac{\partial \tilde v}{\partial n}\right)}_{L^2(\Gamma_K)}\le & C h_{}^{k}\left(\norm{\tilde v}_{H^{k+1}\left(\tilde K'\right)}+\norm{\tilde v}_{H^{k+1}\left(\tilde K\right)}\right)\\
\le& C h_{}^{k}\norm{\tilde v}_{H^{k+1}\left(\tilde K\cup \tilde{ K}'\right)}.
\end{aligned}
\end{equation}
We conclude by summing over $K\in\mathcal G_h$.
\end{proof}

\section{A priori error estimate} \label{sec:estimate}
The preparatory results of Propositions \ref{thmchenonviene} and \ref{approx_physical} were needed in order to prove the following convergence theorem.
\begin{theorem}\label{apriori_error_estimate}
Let $\frac{1}{2}< k\le p$. There exists $\overline\beta>0$ such that, for every $\beta\ge\overline\beta$, if $u\in H^{k+1}(\Omega)$ is the solution to \eqref{3} and $u_h\in \tilde{V}_h$ solution to \eqref{40}, then
\begin{equation}
\norm{u-u_h}_{1,h,\Omega}\le C\left( h_{}^k\norm{u}_{H^{k+1}(\Omega)} +h_{}^r\norm{\tilde u}_{H^{r+1}(S_h)}\right),\label{estimate_speremo}
\end{equation}
where $S_h$ is the strip of width $ch_{}$, $c\ge 1$, such that $S_h\supseteq \bigcup_{K\in \mathcal{M}_h^b}\left(\tilde K\cup \tilde K'\right)$, and $K'\in\mathcal{M}_h^g\cap\mathcal N(K)$. Moreover, \eqref{estimate_speremo} holds for every $r$ such that:
\begin{itemize}
\item $0\le r < p-\frac{1}{2}$ with the stabilization in the parametric domain of Definition \ref{stabilization_parametric};
\item $0\le r \le p$ with the stabilization in the physical domain of Definition \ref{stabilization_physical}.
\end{itemize}
\end{theorem}
\begin{proof}
From Theorems \ref{thm_well_posedness}, \ref{theorem_stability} and \ref{thm1234} we know that $\overline{a}_h(\cdot,\cdot)$ is coercive w.r.t. $\norm{\cdot}_{1,h,\Omega}$, i.e. there exists $\alpha>0$ such that for every $u_h\in \tilde{V}_h$
\begin{equation}\label{passage_coercivity}
\alpha \sup_{\substack{w_h\in \tilde{V}_h\\ w_h\ne 0}}\frac{\overline{a}_h(u_h,w_h)}{\norm{w_h}_{1,h,\Omega}}\ge\norm{u_h}_{1,h,\Omega}.
\end{equation}
Let $v_h\in \tilde{V}_h$. Using the triangular inequality and coercivity, we get:
\begin{equation}\label{first_passage}
\begin{aligned}
\norm{u-u_h}_{1,h,\Omega}\le&\norm{u-v_h}_{1,h,\Omega}+\norm{v_h-u_h}_{1,h,\Omega}\\
\le&\norm{u-v_h}_{1,h,\Omega}+\alpha\sup_{\substack{w_h\in V_h\\ w_h\ne 0}}\frac{\overline{a}_h(v_h-u_h,w_h)}{\norm{w_h}_{1,h,\Omega}}.
\end{aligned}
\end{equation}
Then, recalling that $u_h$ solves \eqref{40}, we get
\begin{equation*}
\begin{aligned}
\overline{a}_h(v_h-u_h,w_h)=&\overline{a}_h(v_h,w_h)-\overline{a}_h(u_h,w_h)=\overline{a}_h(v_h,w_h)-\overline F_h(w_h)\\
=&\int_{\Omega}\nabla v_h\cdot\nabla w_h-\int_{\Gamma_D}R_h(v_h)w_h-\int_{\Gamma_D} v_h R_h(w_h)+\beta\int_{\Gamma_D}\mathsf{h}^{-1}v_hw_h\\
&-\int_{\Omega}fw_h+\int_{\Gamma_D}g_D R_h(w_h)-\beta \int_{\Gamma_D}\mathsf{h}^{-1}g_D w_h.
\end{aligned}
\end{equation*}
Since $u$ solves \eqref{3}: $\int_{\Omega}fw_h=\int_{\Omega}\nabla u\cdot \nabla w_h -\int_{\Gamma_D}\frac{\partial u}{\partial n}w_h$ and $\restr{u}{\Gamma_D}=g_D$, hence:
\begin{equation*}
\begin{aligned}
\overline{a}_h(v_h-u_h,w_h)=&\underbrace{\int_\Omega \nabla(v_h-u)\cdot\nabla w_h}_{\RomanNumeralCaps 1}-\underbrace{\int_{\Gamma_D} (R_h(v_h)-\frac{\partial u}{\partial n})w_h}_{\RomanNumeralCaps 2}\\
&+\underbrace{\int_{\Gamma_D}(u-v_h)R_h(w_h)}_{\RomanNumeralCaps 3}+\underbrace{\beta\int_{\Gamma_D} \mathsf{h}^{-1}(v_h-u)w_h}_{\RomanNumeralCaps 4}.
\end{aligned}
\end{equation*}

Let us now estimate the four terms separately. We will leave $\RomanNumeralCaps 2$ for last since its analysis depends on the choice of the stabilization. Clearly
\begin{equation}
\RomanNumeralCaps 1 + \RomanNumeralCaps 4 \le  C \norm{u-v_h}_{1,h,\Omega}\norm{w_h}_{1,h,\Omega},\label{passage1,4}
\end{equation}
where $C>0$ linearly depends on $\beta$. Note that this will not compromise the uniformity of the resulting constant, provided that $\beta$ is chosen as close as possible to $\overline \beta$ (see the discussion in Remark~\ref{remark_beta}).
Using the stability property \eqref{41} and taking $K'\in\mathcal{N}(K)\cap \mathcal M_h^g$ (if $K$ itself is a good element, then take $K'=K$), we get:
\begin{equation}
\begin{aligned}
\RomanNumeralCaps 3^2\le& \norm{\mathsf{h}^{-\frac{1}{2}}\left(u-v_h\right)}^2_{L^2(\Gamma_D)}\sum_{K\in\mathcal{G}_h}\norm{\mathsf{h}^{\frac{1}{2}}R_h(w_h)}^2_{L^2(\Gamma_K)}\\
\le& \norm{u-v_h}^2_{1,h,\Omega}C\sum_{K\in\mathcal{G}_h} \norm{\nabla w_h}^2_{L^2(K'\cap\Omega)}
\le C \norm{u-v_h}^2_{1,h,\Omega}\norm{ w_h}^2_{1,h,\Omega}\label{passage3}.
\end{aligned}
\end{equation}
Let us estimate the term $\RomanNumeralCaps 2$. By definition of the norm $\norm{\cdot}_{1,h,\Omega}$:
\begin{equation*}
\begin{aligned}
\RomanNumeralCaps 2 
 \le \norm{\mathsf{h}^{\frac{1}{2}}\left(R_h(v_h)-\frac{\partial u}{\partial n}\right)}_{L^2({\Gamma_D})}\norm{w_h}_{1,h,\Omega}.
\end{aligned}
\end{equation*}
Now, we choose $v_h=\Pi_0(\tilde u)$ and distinguish two cases.
\begin{itemize}
\item If we use the stabilization in the parametric domain of Definition \ref{stabilization_parametric}, hence apply Proposition \ref{thmchenonviene}, we get, for any $0\le r<p-\frac{1}{2}$,
\begin{equation*}
\begin{aligned}
\sum_{K\in\mathcal G_h}&\norm{\mathsf{h}^{\frac{1}{2}}\left( R_h\left(\Pi_0\left(\tilde u\right)\right)-\frac{\partial u}{\partial n}\right)}_{L^2(\Gamma_K)}\norm{w_h}_{1,h,\Omega} \\
&=\sum_{K\in\mathcal G_h}\norm{\mathsf{h}^{\frac{1}{2}}\left( R_h\left(\Pi_0\left(\tilde u\right)\right)-\frac{\partial \tilde u}{\partial n}\right)}_{L^2(\Gamma_K)}\norm{w_h}_{1,h,\Omega}\\
&\le \sum_{K\in\mathcal G_h}C\left(h_{}^k\norm{\tilde u}_{H^{k+1}\left(\tilde K\cup\tilde K'\right)}+h^r_{}\norm{\tilde u}_{H^{k+1}\left(\tilde K\cup\tilde K'\right)} \right)\norm{w_h}_{1,h,\Omega},
\end{aligned}
\end{equation*}
\item Employing the stabilization in the physical domain of Definition \ref{stabilization_physical}, hence apply Proposition \ref{approx_physical}, we obtain
\begin{equation*}
\begin{aligned}
\sum_{K\in\mathcal G_h}&\norm{\mathsf{h}^{\frac{1}{2}}\left( R_h\left(\Pi_0\left(\tilde u\right)\right)-\frac{\partial u}{\partial n}\right)}_{L^2(\Gamma_K)}\norm{w_h}_{1,h,\Omega}\\
&=\sum_{K\in\mathcal G_h}\norm{\mathsf{h}^{\frac{1}{2}}\left( R_h\left(\Pi_0\left(\tilde u\right)\right)-\frac{\partial \tilde u}{\partial n}\right)}_{L^2(\Gamma_K)}\norm{w_h}_{1,h,\Omega}\\
&\le \sum_{K\in\mathcal G_h}C h_{}^k\norm{\tilde u}_{H^{k+1}\left(\tilde K\cup\tilde K'\right)}\norm{w_h}_{1,h,\Omega}.
\end{aligned}
\end{equation*}
\end{itemize}
Therefore, we have that
\begin{equation*}
\RomanNumeralCaps 2\le C \left( h_{}^k\norm{ u}_{H^{k+1}(\Omega)}+ h^r_{}\norm{\tilde u}_{H^{k+1}(S_h)} \right)\norm{w_h}_{1,h,\Omega},
\end{equation*}
where $S_h$ is the strip of width $ch_{}$, $c\ge 1$, such that $S_h\supseteq \bigcup_{K\in \mathcal{M}_h^b}\left(\tilde K\cup \tilde K'\right)$, and $K'\in\mathcal{M}_h^g\cap\mathcal N(K)$ and we can choose any $r$ such that:
\begin{itemize}
\item $0\le r<p-\frac{1}{2}$ if we use the stabilization in the parametric  domain, hence apply Proposition \ref{thmchenonviene};
\item $0\le r\le p$ if we use the one in the physical domain and use Proposition \ref{approx_physical}.
\end{itemize}
As a consequence, we have that
\begin{equation}\label{last_inequality}
\begin{aligned}
\overline{a}_h(\Pi_0(\tilde u)-u_h,w_h)\le &\norm{u-\Pi_0(\tilde u)}_{1,h,\Omega}\norm{w_h}_{1,h,\Omega} \\
&+ C \left( h_{}^k\norm{u}_{H^{k+1}(\Omega)}+ h^r_{}\norm{\tilde u}_{H^{k+1}(S_h)} \right)\norm{w_h}_{1,h,\Omega}\\
&+ C\norm{u-\Pi_0(\tilde u)}_{1,h,\Omega}\norm{w_h}_{1,h,\Omega},
\end{aligned}
\end{equation}
where in \eqref{passage1,4}, \eqref{passage3} we choose again $v_h=\Pi_0(\tilde u)$.

We now combine the last inequality \eqref{last_inequality} with \eqref{passage_coercivity} and \eqref{first_passage} to obtain
\begin{equation*}
\begin{aligned}
&\norm{u-u_h}_{1,h,\Omega}\le \norm{u-\Pi_0(\tilde u)}_{1,h,\Omega}+\alpha \sup_{\substack{w_h\in \tilde{V}_h\\ w_h\ne 0}}\frac{\overline{a}_h(\Pi_0(\tilde u)-u_h,w_h)}{\norm{w_h}_{1,h,\Omega}}\\
\le& \left(1+\alpha\left(1+C\right)\right) \norm{u-\Pi_0(\tilde u)}_{1,h,\Omega}
+\alpha C \Big( h_{}^k\norm{u}_{H^{k+1}(\Omega)}
+ h^r_{}\norm{\tilde u}_{H^{k+1}(S_h)} \Big).
\end{aligned}
\end{equation*}
Using approximation results of quasi-interpolants in spline spaces \cite{Buffa2016}, we conclude
\begin{equation*}
\begin{aligned}
\norm{u-u_h}_{1,h,\Omega}\le C\left( h_{}^k\norm{u}_{H^{k+1}(\Omega)} +h_{}^r\norm{\tilde u}_{H^{r+1}(S_h)}\right),
\end{aligned}
\end{equation*}
where $r$ is the same as above.
\end{proof}
\begin{remark}
As already observed in Remark \ref{remark1}, when $u\in H^{k+1}(\Omega)$ with $\frac{1}{2}< k<p-\frac{1}{2}$, both stabilizations give rise to optimal a priori error estimates. When $u\in H^{k+1}(\Omega)$ with $p-\frac{1}{2}\le k\le p$ and $k>\frac{1}{2}$, instead, stabilization in Definition \ref{stabilization_parametric} is sub-optimal. In this case the estimate can be modified and improved using the following result.
\end{remark}
\begin{lemma}\label{lemma_improved_ineq}
Let $\eps>0$ and $S_h$ be defined as in Theorem~\ref{apriori_error_estimate}. Then, there exists $C>0$ such that
\begin{equation*}
\norm{\tilde u}_{H^{r+1}(S_h)}\le C h_{}^{\frac{1}{2}-\eps}\norm{u}_{H^{p+\frac{3}{2}-\eps}(\Omega)}\qquad\forall\ u\in H^{p+\frac{3}{2}-\eps}(\Omega),\ \forall\ 0\le r<p-\frac{1}{2}.
\end{equation*}
\end{lemma}	
\begin{proof}
Using the fact that $r < p$, we are able to recover an integer order for the Sobolev norm and so to apply Lemma~\ref{lemma_ineq} with $s=\frac{1}{2}-\eps$:
\begin{equation*}
\norm{\tilde u}_{H^{r+1}(S_h)}\le \norm{\tilde u}_{H^{p+1}(S_h)}\le C h_{}^{\frac{1}{2}-\eps}\norm{\tilde u}_{H^{p+\frac{3}{2}-\eps}( \Omega_0)}\le C h_{}^{\frac{1}{2}-\eps}\norm{u}_{H^{p+\frac{3}{2}-\eps}( \Omega)}.
\end{equation*}
In the last inequality we used the boundedness of the Sobolev-Stein extension operator.
\end{proof}
\begin{proposition}\label{prop7.4}
Let $u\in H^{p+1}(\Omega)$ be the solution to \eqref{3} and $u_h\in \tilde{V}_h$ solution to \eqref{40}, obtained using the stabilization in the parametric domain of Definition \ref{stabilization_parametric}. Then, the following error estimate holds:
\begin{equation*}
\norm{u-u_h}_{1,h,\Omega}\le C h_{}^{p'} \norm{u}_{H^{p'+\frac{3}{2}}(\Omega)}\qquad\forall\ 0\le p'<p.
\end{equation*}
\end{proposition}
\begin{proof}
It immediately follows combining Theorem \ref{apriori_error_estimate} and Lemma \ref{lemma_improved_ineq}.
\end{proof}
\begin{remark}
At the prize of slightly higher regularity request, optimal convergence rate is to be expected also for stabilization in Definition \ref{stabilization_parametric}.
\end{remark}

\section{Numerical examples}\label{section_numerical_examples}
\subsection{Some details about the implementation}
For accurate numerical integration, we decompose the trimmed elements into smaller quadrilateral tiles where we compute the integrals. These tiles are reparametrized as B\'ezier surfaces of the same degree $p$ as the approximation space used to discretize our PDE, see \cite{pablo} for a detailed explanation. We remark that this reparametrization is also used to compute the boundary integrals. 

In order to compute the stabilization terms appearing in \eqref{40}, first of all for each bad trimmed element $K$ we choose $K'$: among all the neighbours of $K$, we choose (the) one with the largest relative overlap $\abs{K'\cap\Omega}/\abs{K'}$.
Then we need to locally project functions living in $K'$ (or in $Q'$) onto the space of polynomials on $K'$ (or $Q'$) and extend them up to $\Gamma_K$. For the stabilization in the parametric domain, by taking as a basis the Bernstein polynomials on $Q'$ the projection can be computed by knot insertion, while for the stabilization in the physical domain the $L^2$-projection is needed anyhow.

\subsection{Validation of stability} \label{section:square_epsilon}
Let us repeat the numerical experiment of Section \ref{section_stability} in order to validate the effectiveness of our stabilization technique. Let us solve the eigenvalue problem~\eqref{GEPbis} with the stabilization in Definition \ref{stabilization_parametric} (since $\mathbf{F}=\mathbf{Id}$, the two proposed stabilizations techniques are equivalent) in the trimmed domain of Figure \ref{fig:mesh} for the same values of $\eps$ used in Section~\ref{section_stability}. 
The result is shown in Figure~\ref{fig_eigs_good}.

\begin{figure}[!htbp]
	\centering
\subfloat[][Sketch of the setting with exaggerated distances. We plot $\Omega=(0,1)\times(0,0.757)$ (in red), $\{(x,y): y=\overline\xi\}$ (in solid blue), $\{(x,y): y=0.75\}$ (in dotted gray).\label{fig:mesh}]
	{
		\begin{tikzpicture}
	
\definecolor{color2}{rgb}{0.872549019607843,0.019607843137255,0.0819607843137255}
\definecolor{blue}{rgb}{0,0,255}
\draw[dashed] (0,0) -- (1,0) ;
\draw[dashed] (0,0) -- (0,1) ;
\draw[color2, dashed, ultra thick] (0,3+0.6) -- (4,3+0.6);
\draw[ultra thick] (0,0) -- (4,0) -- (4,4) -- (0,4) -- (0,0);
\draw[color2,  ultra thick]  (0,3+0.6)--(0,0)--(4,0)--(4,3+0.6);
\coordinate (a) at (2,2);
\node at (a)   {\huge$\mathbf{\textcolor{color2}{\Omega}}$};
\begin{scope}
  \clip (0,0) -- (4,0) -- (4,4) -- (0,4) -- (0,0);
\end{scope}

\draw[step=2cm,lightgray,very thin] (0,0) grid (4,4);
\draw[lightgray,very thin] (0,1) -- (4,1);
\draw[lightgray,very thin] (1,0) -- (1,4);
\draw[lightgray,very thin] (3,0) -- (3,4);

   \draw[decoration={brace,raise=5pt},decorate]
  (0,3.3) -- node[left=6pt] {\large{$\varepsilon$}} (0,3.3+0.3);
  
  \draw[blue, thick](0,3.3) -- (4,3.3);
  
    \draw[gray, dotted, ultra thick](0,3) -- (4,3);
\end{tikzpicture}
	}
	\hspace{2cm}
	\subfloat[][Maximum eigenvalue vs $\varepsilon$.\label{fig_eigs_good}]
	{
		\includestandalone[width=0.35\textwidth,keepaspectratio=true]{lambda_stab}
	}
	\caption{Testing the stability of formulation \eqref{40} with respect to trimming. }
\end{figure}
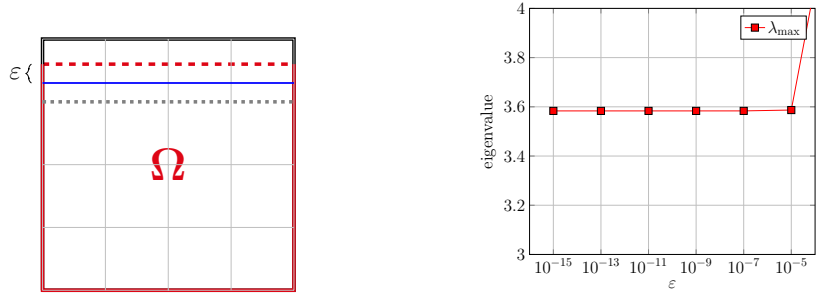

This time we observe that the spectrum remains bounded independently of $\eps$, confirming our method to be stable.
\subsection{Validation of the a priori error estimate}\label{sec:apriori}
In the following we focus on the Poisson problem \eqref{40} with the difference that, while we impose Dirichlet boundary conditions weakly on the trimmed parts of the boundary, on the other parts where the mesh is fitted with the boundary we impose them in the strong sense.

\textbf{Test 1.} Let $\Omega=\Omega_0\setminus\overline\Omega_1$ be defined as in Figure \ref{domain_fig2bis}, where $\Omega_0 = {\bf F}((0,1)^2)$ is a quarter of annulus ($\F$ is non linear) constructed with biquadratic NURBS, and $\Omega_1$ is the image of a ball in the parametric domain through the isogeometric map, namely  $\Omega_1=\F(B(0,r))$, with $r=0.76$. We consider as manufactured solution
$
u_{ex}(x,y)=e^x\sin(xy).
$
We solve the Poisson problem using the stabilized formulation \eqref{40}, the stabilization in the parametric domain and the parameters $\beta=1$ and $\theta=0.1$. The results of convergence for different values of $p$, that are displayed Figure \ref{order_conv_fig2bis}, show that we obtain the optimal order of convergence.

\begin{figure}[!ht]
\centering
\subfloat[][Trimmed domain.]
{
  \begin{tikzpicture}
\definecolor{color2}{rgb}{0,0,0}
\draw(2,0) -- (4,0) ;
\draw(0,2) -- (0,4) ;
\coordinate (a) at (1,3);
\node at (a)   {\large$\mathbf{\Omega}$};
\draw[color2,ultra thick]  (0,2) -- (0,4);
\begin{scope}
  \clip(0,0) -- (4.1,0) -- (4.1,4.1) -- (0,4.1) -- (0,0);
   \draw[ultra thick,domain=0.8:4,smooth,variable=\x,dashed,color2]  plot ({\x},{ -\x^2 + 3.24308*\x - 0.12125)});
  \draw (0,0) circle (2cm);
    \draw[color2,ultra thick]  (0.8,1.81) arc (65:90:2);
  \foreach \x in  {0,25,50,75}
    {\draw[ultra thin, lightgray]  (0,0) circle (2.\x cm);
    \draw[ultra thin, lightgray]  (0,0) circle (3.\x cm);}
    \draw[ultra thick,color2] (0,0) circle (4cm);
    \draw[color2,ultra thick]  (3.2,0) -- (4,0);
\foreach \ang in {0,2,4,6,8,10,12,14,16,18,20,22,24,26,28,31} {
  \draw [ultra thin, lightgray] (\ang * 180 / 30:4) -- (\ang * 180 / 30:2);
}
\end{scope}

\draw[step= 0.385cm,lightgray,ultra thin] (-5,0.5) grid (-2,3.5);
\draw [->,ultra thick] (-1.75,2) to [bend left=40] node[above,scale=1.5] {\small$F$} (-0.25,2);
\draw[black](-5,0.5) -- (-2,0.5) -- (-2,3.5) -- (-5,3.5) -- (-5,0.5);
\draw[color=color2,ultra thick](-3,0.5) -- (-2,0.5) -- (-2,3.5) -- (-5,3.5) -- (-5,2.5);
\node  at (-3.5,2.5) {\large$\mathbf{\widehat{\Omega}}$};
\begin{scope}
  \clip (-5,-2) -- (-3,-2) -- (-3,3) -- (-5,3) -- (-5,-2);
    \draw[ultra thick,color2, dashed] (-5,2.5)  arc (90:0:2);
\end{scope}
\end{tikzpicture}\label{domain_fig2bis}
  }
  \subfloat[][Convergence rates.]
{
  \newcommand{\logLogSlopeTriangleOne}[5]
{

    \pgfplotsextra
    {
        \pgfkeysgetvalue{/pgfplots/xmin}{\xmin}
        \pgfkeysgetvalue{/pgfplots/xmax}{\xmax}
        \pgfkeysgetvalue{/pgfplots/ymin}{\ymin}
        \pgfkeysgetvalue{/pgfplots/ymax}{\ymax}

        \pgfmathsetmacro{\xArel}{#1}
        \pgfmathsetmacro{\yArel}{#3}
        \pgfmathsetmacro{\xBrel}{#1-#2}
        \pgfmathsetmacro{\yBrel}{\yArel}
        \pgfmathsetmacro{\xCrel}{\xArel}

        \pgfmathsetmacro{\lnxB}{\xmin*(1-(#1-#2))+\xmax*(#1-#2)} 
        \pgfmathsetmacro{\lnxA}{\xmin*(1-#1)+\xmax*#1} 
        \pgfmathsetmacro{\lnyA}{\ymin*(1-#3)+\ymax*#3} 
        \pgfmathsetmacro{\lnyC}{\lnyA+#4*(\lnxA-\lnxB)}
        \pgfmathsetmacro{\yCrel}{\lnyC-\ymin)/(\ymax-\ymin)} 

        \coordinate (A) at (rel axis cs:\xArel,\yArel);
        \coordinate (B) at (rel axis cs:\xBrel,\yBrel);
        \coordinate (C) at (rel axis cs:\xCrel,\yCrel);

        \draw[#5]   (A)-- node[pos=0.5,anchor=north] {1} 
                    (B)-- 
                    (C)-- node[pos=0.5,anchor=east] {#4} 
                    cycle;
    }
}


\begin{tikzpicture}

\definecolor{color1}{rgb}{255, 0, 255}
\definecolor{color2}{rgb}{0,0,1}
\definecolor{color3}{rgb}{1,0,0}
\definecolor{color4}{rgb}{0,255,0}

\begin{loglogaxis}[
        scale only axis, 
        xlabel={$h$},
        ylabel={energy norm},
        ymin=1e-10,
        ymax = 1e2,
        legend pos=north west,
	        xtick = {0.5,0.25,0.125,0.0675,0.031252,0.015625,0.0078125},
      xticklabels={ $1/2$, $1/4$, $1/8$, $1/16$, $1/32$,$1/64$, $1/128$},
       ytick={ 1e-8,1e-6,1e-4,1e-2,1e0,1e2},
    yticklabels={ $10^{-8}$, $10^{-6}$, $10^{-4}$, $10^{-2}$, $10^{0}$,$10^{2}$},
        grid=major,
                   xlabel style={font=\Large},
        ylabel style={font=\Large},
      ]

     \addplot[semithick, color1, mark=triangle, mark size=3, mark options={draw=color1}] table[ x index = 0, y index = 1]{errNitsche_ringhole_parametric_0.1.txt};
    \addlegendentry{$p=2$};
        \addplot[semithick, color2, mark=square, mark size=3, mark options={draw=color2}] table[ x index = 0, y index = 2] {errNitsche_ringhole_parametric_0.1.txt};
    \addlegendentry{$p=3$};
         \addplot[semithick, color3, mark=o, mark size=3, mark options={draw=color3}] table[ x index = 0, y index = 3]{errNitsche_ringhole_parametric_0.1.txt};
    \addlegendentry{$p=4$};
            \addplot[semithick, color4, mark=diamond, mark size=3, mark options={draw=color4}] table[ x index = 0, y index = 4] {errNitsche_ringhole_parametric_0.1.txt};
    \addlegendentry{$p=5$};
\logLogSlopeTriangleOne{0.075}{-0.12}{0.64}{2}{black, semithick} ;
\logLogSlopeTriangleOne{0.075}{-0.12}{0.48}{3}{black, semithick} ;
 \logLogSlopeTriangleOne{0.075}{-0.12}{0.35}{4}{black, semithick} ;
 \logLogSlopeTriangleOne{0.075}{-0.12}{0.23}{5}{black, semithick} ;
  \end{loglogaxis}

\end{tikzpicture}\label{order_conv_fig2bis}
  }
  \caption{Geometry and convergence rates for the quarter of annulus with hole.}
\end{figure}
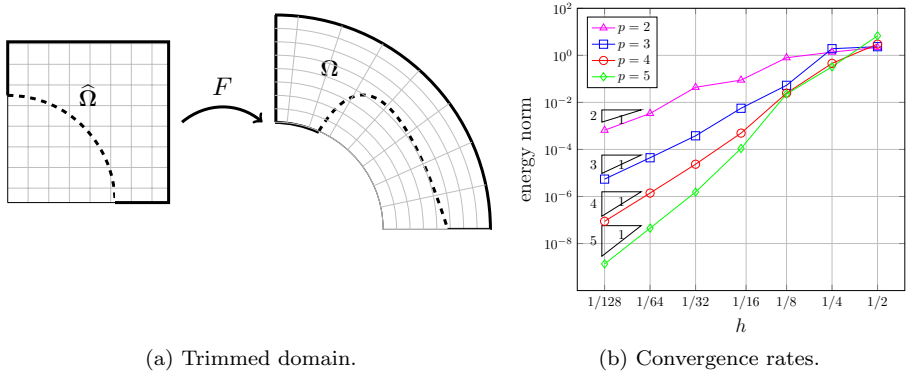

\textbf{Test 2.} We now consider the Poisson problem in the L-shaped domain shown in Figure \ref{Lshaped}, given by $\Omega=\Omega_0\setminus\overline\Omega_1$, where $\Omega_0=(-2,1)\times (-1,2)$ and $\Omega_1=(0,1)\times(-1,0)$. The exact solution is chosen as the singular function that, in polar coordinates, reads as $u(r,\varphi)=r^{\frac{2}{3}}\sin\left( \frac{2}{3}\varphi\right)\in H^{\frac{5}{3}-\delta}(\Omega)$, for every $\delta>0$. The function has a singularity at the re-entrant corner in the origin, and the domain is chosen in such a way that the corner is always located in the interior of an element. We employ the formulation \eqref{40} together with the stabilization operator in Definition \ref{stabilization_parametric}, noting that since the parametrization is a simple scaling, both stabilizations are equivalent. This time we set the parameters $\theta=1$ and, due to the presence of the singularity, $\beta=(p+1)\cdot 10$. The numerical results of Figure \ref{Lshapedconv} agree with the theory as the method converges with order $\frac{2}{3}$, and the sub-optimal behaviour is due to the low regularity of the reference solution.

\begin{figure}[!ht]
\centering
\subfloat[][Trimmed domain.]
{
 \begin{tikzpicture}
\definecolor{color2}{rgb}{0,0,0}

\draw[step=0.75cm,lightgray,very thin] (0,0) grid (6,6);

\draw[dashed,ultra thick] (4,0) -- (4,2);
\draw[dashed,ultra thick] (4,2) -- (6,2);

\draw[ultra thick] (6,2) -- (6,6) -- (0,6) -- (0,0) -- (4,0);
\coordinate (a) at (3,3);
\node at (a)   {\huge$\mathbf{\Omega}$};
\end{tikzpicture}\label{Lshaped}
  }
  \subfloat[][Convergence rates.]
{
  \newcommand{\logLogSlopeTriangleOne}[5]
{

    \pgfplotsextra
    {
        \pgfkeysgetvalue{/pgfplots/xmin}{\xmin}
        \pgfkeysgetvalue{/pgfplots/xmax}{\xmax}
        \pgfkeysgetvalue{/pgfplots/ymin}{\ymin}
        \pgfkeysgetvalue{/pgfplots/ymax}{\ymax}

        \pgfmathsetmacro{\xArel}{#1}
        \pgfmathsetmacro{\yArel}{#3}
        \pgfmathsetmacro{\xBrel}{#1-#2}
        \pgfmathsetmacro{\yBrel}{\yArel}
        \pgfmathsetmacro{\xCrel}{\xArel}

        \pgfmathsetmacro{\lnxB}{\xmin*(1-(#1-#2))+\xmax*(#1-#2)} 
        \pgfmathsetmacro{\lnxA}{\xmin*(1-#1)+\xmax*#1} 
        \pgfmathsetmacro{\lnyA}{\ymin*(1-#3)+\ymax*#3} 
        \pgfmathsetmacro{\lnyC}{\lnyA+#4*(\lnxA-\lnxB)}
        \pgfmathsetmacro{\yCrel}{\lnyC-\ymin)/(\ymax-\ymin)} 

        \coordinate (A) at (rel axis cs:\xArel,\yArel);
        \coordinate (B) at (rel axis cs:\xBrel,\yBrel);
        \coordinate (C) at (rel axis cs:\xCrel,\yCrel);

        \draw[#5]   (A)-- node[pos=0.5,anchor=north] {1} 
                    (B)-- 
                    (C)-- node[pos=0.5,anchor=east] {#4} 
                    cycle;
    }
}


\begin{tikzpicture}

\definecolor{color1}{rgb}{255, 0, 255}
\definecolor{color2}{rgb}{0,0,1}
\definecolor{color3}{rgb}{1,0,0}
\definecolor{color4}{rgb}{0,255,0}

\begin{loglogaxis}[
        scale only axis, 
        xlabel={$h$},
        ylabel={energy norm},
        ymin=1e-2,
        ymax = 5,
        legend pos=north west,
	        xtick = {0.5,0.25,0.125,0.0675,0.031252,0.015625,0.0078125},
      xticklabels={ $1/2$, $1/4$, $1/8$, $1/16$, $1/32$,$1/64$, $1/128$},
       ytick={ 1e-8,1e-6,1e-4,1e-3,1e-2,1e-1,1e0, 1e1,1e2},
    yticklabels={ $10^{-8}$, $10^{-6}$, $10^{-4}$,$10^{-3}$, $10^{-2}$, $10^{-1}$, $10^{0}$,$10^{1}$,$10^{2}$},
        grid=major,
                   xlabel style={font=\Large},
        ylabel style={font=\Large},
      ]

  \addplot[semithick, color1, mark=triangle, mark size=3, mark options={draw=color1}] table[ x index = 0, y index = 1] {errNitsche_Lshaped_parametric_01.txt};
    \addlegendentry{$p=2$};
      \addplot[semithick, color2, mark=square, mark size=3, mark options={draw=color2}] table[ x index = 0, y index = 2]  {errNitsche_Lshaped_parametric_01.txt};
    \addlegendentry{$p=3$};
      \addplot[semithick, color3, mark=o, mark size=3, mark options={draw=color3}] table[ x index = 0, y index = 3] {errNitsche_Lshaped_parametric_01.txt};
    \addlegendentry{$p=4$};
           \addplot[semithick, color4, mark=diamond, mark size=3, mark options={draw=color4}] table[ x index = 0, y index = 4] {errNitsche_Lshaped_parametric_01.txt};
    \addlegendentry{$p=5$};
\logLogSlopeTriangleOne{0.09}{-0.12}{0.36}{2/3}{black, semithick} ;
  \end{loglogaxis}

\end{tikzpicture}\label{Lshapedconv}
  }
  \caption{Geometry description and convergence rates for the L-shaped domain.}
\end{figure}
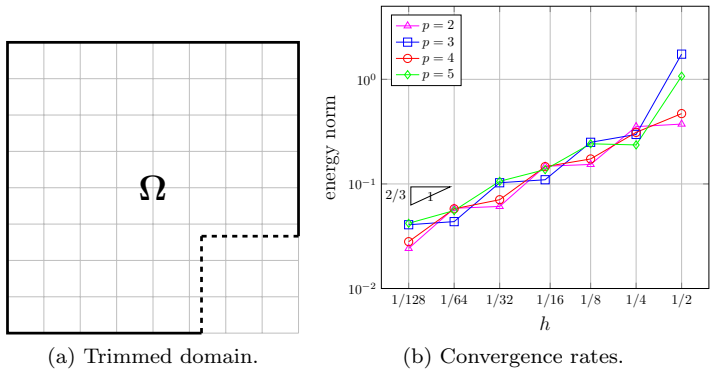

%

%

\textbf{Test 3.} The goal of this test is to show that, when the regularity of the mapping $\F$ is low between trimmed elements and their neighbors, the stabilization in the physical domain is more effective than the ones based on polynomial extensions in the parametric domain (as it is the case for our stabilization in the parametric domain, but also for the method proposed in \cite{MARUSSIG201879}).
Let us consider again as the domain $\Omega_0$ the quarter of annulus, this time parametrized with a different map ${\bf F}$: starting from the standard biquadratic NURBS parametrization, we perform knot insertion adding the knot $\xi = 0.75$, with multiplicity $2$, in the direction corresponding to the angular coordinate, that corresponds to the thick black line in Figure~\ref{quarter_mod}. In order to get a geometry of class $C^0$, we set the second coordinate of one control point, highlighted in Figure~\ref{quarter_mod_ctrl}, equal to $0.5$ in homogeneous coordinates.
\begin{figure}[ht]
\centering
\subfloat[][Trimming line in red and dashed, $C^0$ knot line in thick black.]
{
   \includegraphics[width=0.45\textwidth,keepaspectratio=true]{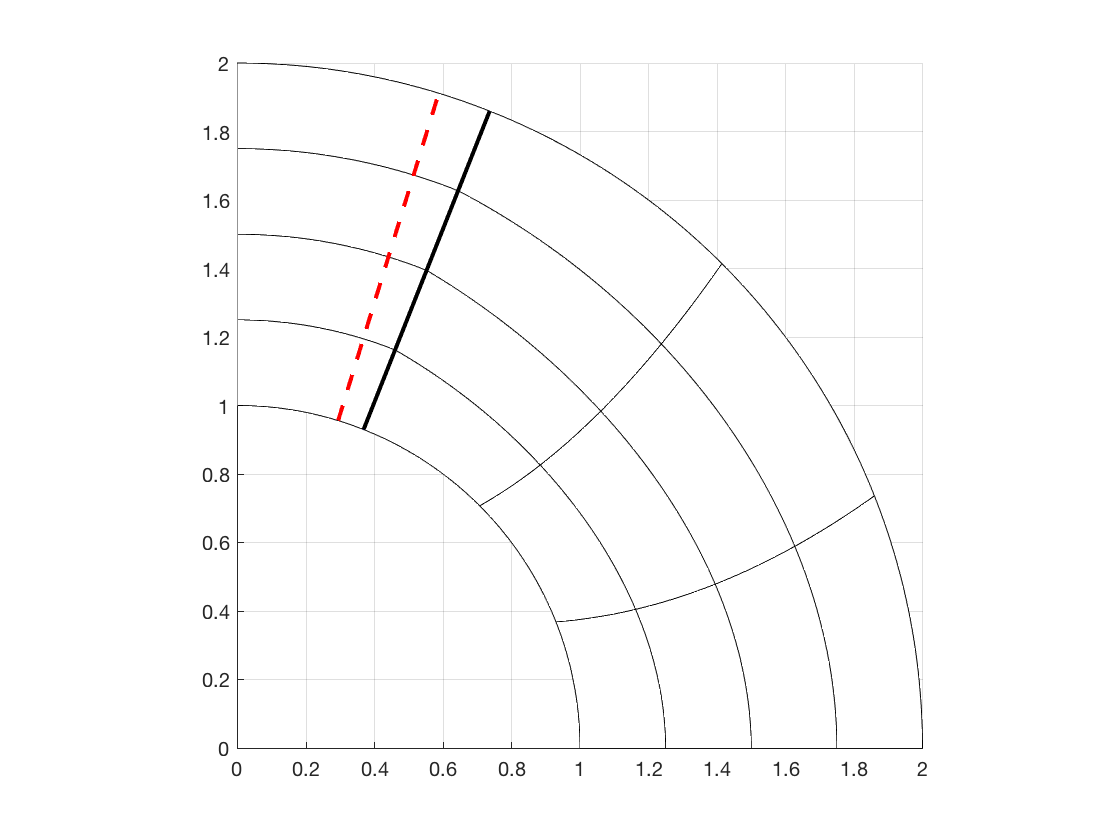}\label{quarter_mod}
  }
  \subfloat[][Control points.]
{
  \includegraphics[width=0.45\textwidth,keepaspectratio=true]{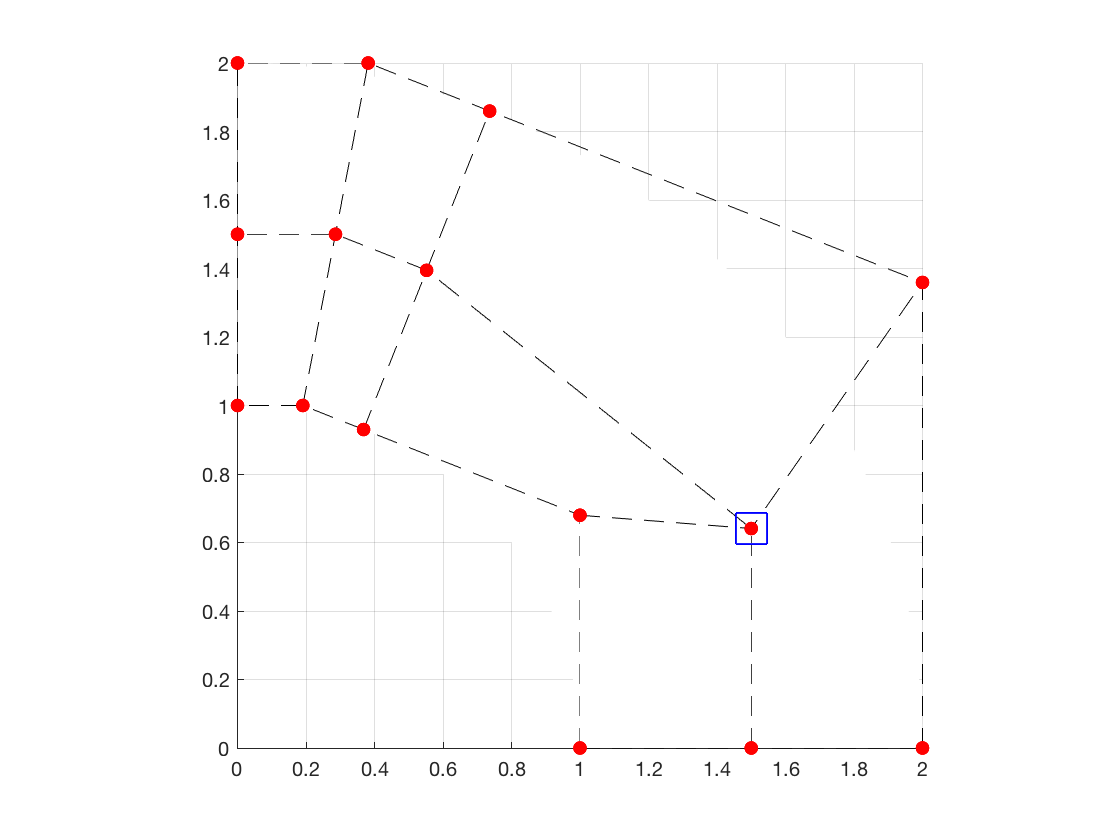}\label{quarter_mod_ctrl}
  }
  \caption{Lower inter-regularity parametrization of the quarter of annulus.}
\end{figure}
Note that the new parametrization is only of class $C^0$ in correspondence of the knot line given by $\F(\{(x,y): x\in(0,1),\, y=0.75 \})$. To ensure that this knot line is located between $K$ and $K'$, we define the trimmed domain as $\Omega =\F\left( (0,1)\times (0,0.75+\eps)\right)$, with $\eps=10^{-8}$. Here we set $\theta=1$ and, because of the lower regularity of the parametrization, $\beta=(p+1)\cdot 25$.
 We know from Remarks~\ref{remark1} and~\ref{remark6.7} that the convergence rate deriving from the stabilization in Definition~\ref{stabilization_parametric} (and any stabilization based on polynomial extensions in the parametric domain) may suffer of sub-optimality. In particular, from Figure \ref{Lshaped_conv_param}, we see that the error with the stabilization in the parametric domain is converging just as $h^{\frac{1}{2}}$ for any degree $p$, while in Figure \ref{Lshaped_conv_phys} we observe that the desired convergence rates are reached when using the stabilization in the physical domain.
\begin{figure}[!ht]
\centering
\subfloat[][Stabilization in Definition \ref{stabilization_parametric}.]
{
   \includegraphics[width=0.4\textwidth,keepaspectratio=true]{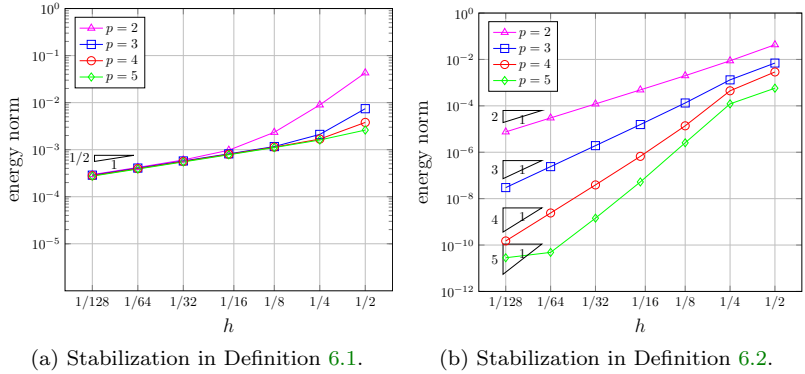}\label{Lshaped_conv_param}
  }
  \subfloat[][Stabilization in Definition \ref{stabilization_physical}.]
{
  \newcommand{\logLogSlopeTriangleOne}[5]
{

    \pgfplotsextra
    {
        \pgfkeysgetvalue{/pgfplots/xmin}{\xmin}
        \pgfkeysgetvalue{/pgfplots/xmax}{\xmax}
        \pgfkeysgetvalue{/pgfplots/ymin}{\ymin}
        \pgfkeysgetvalue{/pgfplots/ymax}{\ymax}

        \pgfmathsetmacro{\xArel}{#1}
        \pgfmathsetmacro{\yArel}{#3}
        \pgfmathsetmacro{\xBrel}{#1-#2}
        \pgfmathsetmacro{\yBrel}{\yArel}
        \pgfmathsetmacro{\xCrel}{\xArel}

        \pgfmathsetmacro{\lnxB}{\xmin*(1-(#1-#2))+\xmax*(#1-#2)} 
        \pgfmathsetmacro{\lnxA}{\xmin*(1-#1)+\xmax*#1} 
        \pgfmathsetmacro{\lnyA}{\ymin*(1-#3)+\ymax*#3} 
        \pgfmathsetmacro{\lnyC}{\lnyA+#4*(\lnxA-\lnxB)}
        \pgfmathsetmacro{\yCrel}{\lnyC-\ymin)/(\ymax-\ymin)} 

        \coordinate (A) at (rel axis cs:\xArel,\yArel);
        \coordinate (B) at (rel axis cs:\xBrel,\yBrel);
        \coordinate (C) at (rel axis cs:\xCrel,\yCrel);

        \draw[#5]   (A)-- node[pos=0.5,anchor=north] {1} 
                    (B)-- 
                    (C)-- node[pos=0.5,anchor=east] {#4} 
                    cycle;
    }
}


\begin{tikzpicture}

\definecolor{color1}{rgb}{255, 0, 255}
\definecolor{color2}{rgb}{0,0,1}
\definecolor{color3}{rgb}{1,0,0}
\definecolor{color4}{rgb}{0,255,0}

\begin{loglogaxis}[
        scale only axis, 
        xlabel={$h$},
        ylabel={energy norm},
        ymin=1e-12,
        ymax = 1e0,
        legend pos=north west,
	        xtick = {0.5,0.25,0.125,0.0675,0.031252,0.015625,  0.0078125},
      xticklabels={ $1/2$, $1/4$, $1/8$, $1/16$, $1/32$,$1/64$, $1/128$},
       ytick={ 1e-12,1e-10,1e-8,1e-6,1e-4,1e-2, 1e0},
    yticklabels={ $10^{-12}$, $10^{-10}$, $10^{-8}$, $10^{-6}$, $10^{-4}$,$10^{-2}$,$10^{0}$},
        grid=major,
                   xlabel style={font=\Large},
        ylabel style={font=\Large},
      ]

   \addplot[semithick, color1, mark=triangle, mark size=3, mark options={draw=color1}] table[ x index = 0, y index = 1] {errNitsche_ring_horiz_cut_mod_physical_1.txt};
    \addlegendentry{$p=2$};
      \addplot[semithick, color2, mark=square, mark size=3, mark options={draw=color2}] table[ x index = 0, y index = 2] {errNitsche_ring_horiz_cut_mod_physical_1.txt};
    \addlegendentry{$p=3$};
       \addplot[semithick, color3, mark=o, mark size=3, mark options={draw=color3}] table[ x index = 0, y index = 3]  {errNitsche_ring_horiz_cut_mod_physical_1.txt};
    \addlegendentry{$p=4$};
        \addplot[semithick, color4, mark=diamond, mark size=3, mark options={draw=color4}] table[ x index = 0, y index = 4]{errNitsche_ring_horiz_cut_mod_physical_1.txt};
    \addlegendentry{$p=5$};
\logLogSlopeTriangleOne{0.075}{-0.12}{0.65}{2}{black, semithick} ;
\logLogSlopeTriangleOne{0.075}{-0.12}{0.47}{3}{black, semithick} ;
 \logLogSlopeTriangleOne{0.075}{-0.12}{0.30}{4}{black, semithick} ;
 \logLogSlopeTriangleOne{0.075}{-0.12}{0.17}{5}{black, semithick} ;
  \end{loglogaxis}

\end{tikzpicture}\label{Lshaped_conv_phys}
  }
  \caption{Comparison of the two stabilizations when $\F$ has lower regularity.}
\end{figure}
\subsection{Conditioning}\label{subsection_ns_cond}
Even if an exhaustive discussion about the conditioning of the stiffness matrix in trimmed geometries is beyond the scope of this work (for a more detailed discussion on the topic see, for instance, \cite{DEPRENTER2017297,Prenter:2020aa}), we would like to present some numerical experiments for the sake of completeness. We focus again on the formulation \eqref{40} of the Poisson problem. Again, we impose Dirichlet boundary conditions weakly on the trimmed parts of the boundary, and strongly on the fitted parts.

\textbf{Test 1.}
Let us come back to the quarter of annulus with a hole and, as above, we employ B-splines of degree $p=3$. In Figure~\ref{fig_cond_final1} we show that our stabilization coupled with a simple diagonal scaling, which can be interpreted as a left-right Jacobi preconditioner, is able to solve the conditioning issue. In Figure~\ref{fig_cond_final2} we compare the effectiveness of the diagonal rescaling with and without the stabilization, and we observe that the effect of the stabilization is marginal with respect to the one of the diagonal preconditioner. The stabilization used is the one in the parametric domain with $\beta=1$ and $\theta=0.1$.


\begin{figure}[!ht]
\centering
\subfloat[][]
   {
   \includegraphics[width=0.4\textwidth,keepaspectratio=true]{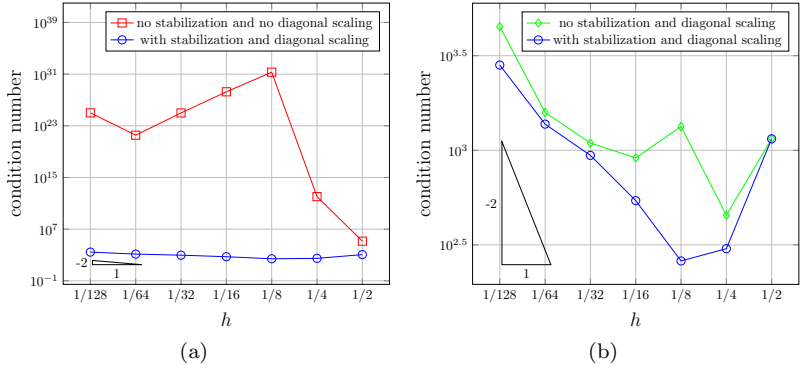}\label{fig_cond_final1}
  }
\subfloat[][]
   {
   \includestandalone[width=0.4\textwidth,keepaspectratio=true]{cond_numb_pc_ring_comparison}\label{fig_cond_final2}
   }
   \caption{Condition number versus $h$ in quarter of annulus geometry.}
\end{figure}

\textbf{Test 2.} Let us consider the same configuration as in the test of Figure~\ref{domain_eps}, for which we notice again that the two stabilizations are equivalent. Let us take B-splines of degree $p=3$, as mesh size $h=2^{-5}$, and set the penalization parameter $\beta=1$. After a simple diagonal rescaling as preconditioner, we compare the condition number of the stiffness matrix, as a function of $\eps$, obtained for the non-stabilized ($\theta=0$) and the stabilized ($\theta=1$) formulations. Note that as the ratio in Definition \ref{def_goodandbad} is the same for all cut elements, it is sufficient to consider only these two values of $\theta$.
The results in Figure \ref{figure_cond_vs_eps} show the diagonal rescaling is acting as a robust preconditioner with respect to the size of the trimming. 
Then, we perform uniform dyadic refinement and we plot the condition number as a function of the mesh-size $h$, obtaining the plots in Figures \ref{jakko1} and \ref{jakko2}.
The results suggest a better behaviour of the condition number when a stabilized formulation is employed to solve the problem.  
\begin{figure}[!ht]
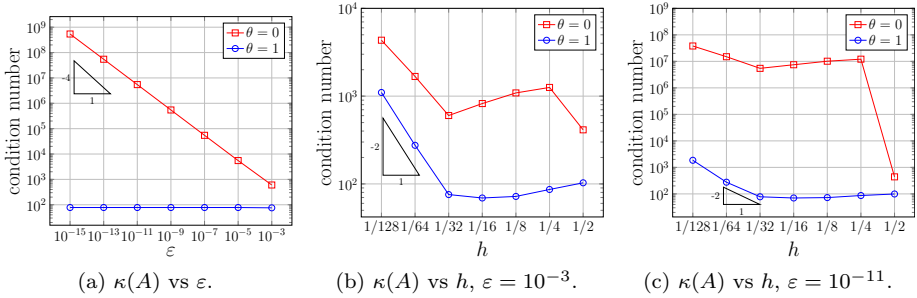

	\centering
	\subfloat[][$\kappa(A)$ vs $\eps$.]
	{
		\includegraphics[width=0.3\textwidth,keepaspectratio=true]{cond_numb_vs_eps.tex}  \label{figure_cond_vs_eps}
	}
	\subfloat[][$\kappa(A)$ vs $h$, $\eps=10^{-3}$.]
	{
		\includegraphics[width=0.3\textwidth,keepaspectratio=true]{cond_numb_vs_theta_eps0,001.tex}\label{jakko1}
	}
	\subfloat[][$\kappa(A)$ vs $h$, $\eps=10^{-11}$.]
	{
		\includestandalone[width=0.3\textwidth,keepaspectratio=true]{cond_numb_vs_theta_eps1e-11}\label{jakko2}
	}
	\caption{Condition number study in the domain of Figure \ref{domain_eps}.} 
\end{figure}  

\textbf{Test 3.} This test is inspired by \cite{DEPRENTER2017297}.  Let us embed $\Omega = \left(0.19,0.78\right) \times \left(0.22,0.78\right)$ in the untrimmed domain $\Omega_0=(0,1)^2$ with un underlying mesh of size $h=2^{-3}$. We consider B-splines of degree $p=2$. Now, let us rotate $\Omega$ around its barycenter for different angles $\alpha$ (see Figure~\ref{rot_square}). For each $\alpha=i\frac{\pi}{200}$, $i=0,\dots,100$ we face a specific trimming configuration where there may appear B-splines whose support intersects in a ``pathological way'' the domain $\Omega$. Let us denote the ``smallest volume fraction'' $\eta:=\min_{K\in\mathcal G_h}\abs{\Omega\cap K}$. In Figure~\ref{fig:rot_square_cond} we plot the condition number of the stiffness matrix against  the smallest volume fraction, in order to compare the non stabilized case with the stabilized (with parameter $\theta=0.5$) and diagonally rescaled one. Let us observe that even if the behaviour of the condition number appears to be much better after stabilization and diagonal rescaling, it is still strongly affected from the way the mesh is cut by the trimming boundary. In this regard this is a counter-example to the fact that diagonal rescaling, together with our stabilization, is a robust preconditioner with respect to the trimming operation. This will be object of further investigations by the authors in the future.
	\begin{figure}[!ht]
	\centering	
	\subfloat[][$\alpha=0, \frac{\pi}{10}, \frac{\pi}{5}, \frac{3\pi}{10},\frac{2\pi}{5},\frac{\pi}{2}$.]
	{
		\includegraphics[scale=.3]{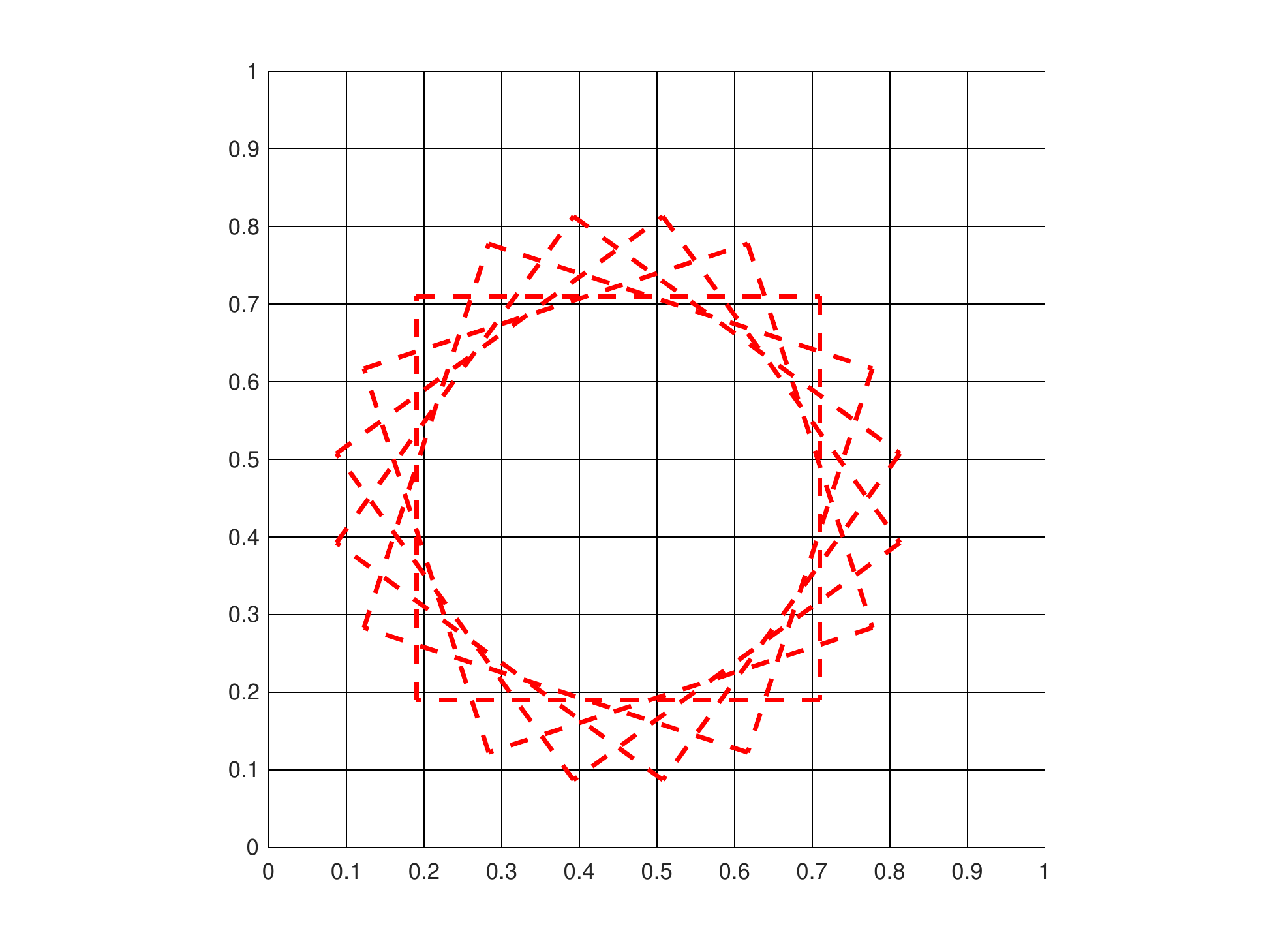}\label{rot_square}
	}
\subfloat[][Condition number vs $\eta$.]
{
\includegraphics[scale=.5]{cond_numb_vs_eta_rotating_square.tex}\label{fig:rot_square_cond}
}
\caption{Condition number the \emph{rotating square}.}
\end{figure}

\appendix

\section{Auxiliary theoretical results} \label{sec:appendix}
\begin{lemma}\label{lemma_hansbo}
	There exists $C>0$ depending on $\Gamma_D$, but independent of the mesh-boundary intersection, such that for every $K\in\mathcal{M}_{0,h}$
	\begin{equation*}
	\norm{v}^2_{L^2(\Gamma_K)}\le C\left( \norm{\mathsf{h}^{-\frac{1}{2}}v}^2_{L^2(K)}+\norm{\mathsf{h}^{\frac{1}{2}}\nabla v}^2_{L^2(K)}\right)\qquad\forall\ v\in H^1(K).
	\end{equation*}
\end{lemma}
\begin{proof}
	It follows straightforward from Lemma 3 in \cite{Hansbo434217}.
\end{proof}
\begin{corollary}\label{corollary_trace}
	There exists $C>0$ depending on $\Gamma_D$, but independent of the mesh-boundary intersection, such that for every $K\in\mathcal{M}_{0,h}$
	\begin{equation*}
	\norm{\frac{\partial v_h}{\partial n}}_{L^2(\Gamma_K)}\le C \norm{\mathsf{h}^{-\frac{1}{2}}\nabla v_h}_{L^2(K)}\qquad\forall\ v_h\in \tilde{V}_h.
	\end{equation*}
\end{corollary}
\begin{proof}
	Let us apply Lemma~\ref{lemma_hansbo}:
	\begin{equation*}
	\norm{\frac{\partial v_h}{\partial n}}^2_{L^2(\Gamma_K)}\le C \left( \norm{\mathsf{h}^{-\frac{1}{2}}\nabla v_h}^2_{L^2(K)}+\norm{\mathsf{h}^{\frac{1}{2}}D^2 v_h}^2_{L^2(K)}\right)\qquad\forall\ v_h\in \tilde{V}_h.
	\end{equation*}
	By a standard inverse inequality, see \cite{BAZILEVS_1}, we get
	\begin{equation*}
	\norm{D^2 v_h}_{L^2(K)}^2\le C\norm{ \mathsf{h}^{-1}\nabla v_h}^2_{L^2(K)}\qquad\forall\ v_h\in \tilde{V}_h,
	\end{equation*}
	where $C>0$ depends on the shape regularity constant of the un-trimmed mesh $\mathcal{M}_{0,h}$. 	
\end{proof}	
\begin{lemma}\label{51}
	Let $Q,Q'\in\hat{\mathcal{M}}_{0,h}$ be neighbor elements in the sense of Definition~\ref{def_goodandbad}. There exists $C>0$ such that
	\begin{equation*}
	\norm{p}_{L^\infty(Q)}\le C\norm{p}_{L^\infty(Q')}\qquad\forall\ p\in\mathbb{Q}_k(\R^d),
	\end{equation*}
	where $C$ depends on $k$, on the shape regularity of the mesh and on the distance between $Q$ and $Q'$.
\end{lemma}
\begin{proof}
The proof follows by a standard \emph{scaling argument} (see \cite{qvalli} for instance).
\end{proof}
The next one says that the $L^2$ norm on the cut portion of an element $Q$ controls the $L^\infty$ (and hence any other) norm on the whole element with an equivalence constant depending on the relative measure of the cut portion.
\begin{lemma}\label{50}
	Let $\theta\in (0,1]$. There exists $C>0$ such that for every $Q\in\hat{\mathcal{M}}_{0,h}$ and every $S\subset Q$ measurable such that $\abs{S} \ge \theta\abs{Q} $, we have
	\begin{equation*}
	\norm{p}_{L^\infty(Q)}\le C h^{-\frac{d}{2}}\norm{p}_{L^2(S)}\qquad\forall\ p\in\mathbb{Q}_k(\R^d),
	\end{equation*}
	where $C$ depends only on $\theta$, $k$ and the mesh regularity.
\end{lemma}
\begin{proof}
	See Proposition 1 in \cite{lozinski}. 
\end{proof}
\begin{lemma}[Hardy's inequality, \cite{brezis2010functional}
	]\label{hardy_ineq}
	Let  $\Omega\subset\R^d$ be a bounded open set of class $C^1$. Then there is a constant $C > 0$ such that
	\begin{equation}
	\norm{\frac{u}{d}}_{L^2(\Omega)}\le C \norm{\nabla u}_{L^2(\Omega)}\qquad\forall\ u\in H^{1}_{0}(\Omega),
	\end{equation}
	where $\text{d}(x):=\operatorname{dist}(x,\Gamma)$.
\end{lemma}
\begin{remark}
	Viceversa,  it is possible to characterize functions in $ H^1_0(\Omega)$ as functions in $H^1(\Omega)$ such that $\frac{u}{d}\in L^2(\Omega)$ (\cite{brezis2010functional}).
\end{remark}
\begin{lemma}\label{lemma_ineq}
	Let $\Omega_1\subset\Omega$ with boundary $\Gamma_1$ such that $\Omega_1=\{x\in\Omega: \operatorname{dist}(x,\Gamma)\ge Ch \}$, where $C\ge 1$ fixed and $\operatorname{dist}(\Gamma,\Gamma_1)\le Ch$. It holds that
	\begin{equation*}
	\norm{v}_{L^2(\Omega\setminus\overline\Omega_1)}\le C h^s \norm{v}_{H^s_i(\Omega)}\qquad\forall\ v\in H^s_i(\Omega),
	\end{equation*}
	where the interpolation space $H^s_i(\Omega)$ or $\left(H^1_0(\Omega),L^2(\Omega) \right)_{s,2}$ is isomorphic to $H^s(\Omega)$ for $0\le s<\frac{1}{2}$, to $H^{\frac{1}{2}}_{00}(\Omega)$ for $s=\frac{1}{2}$ and to $H^s_0(\Omega)$ for $\frac{1}{2}<s\le 1$ (see \cite{tartar}). 
\end{lemma}

\begin{proof}
	We prove the following (like in \cite{lew_negri_2011}):
	\begin{equation}\label{eq3}
	\norm{v}_{L^2(\Omega\setminus\Omega_1)}\le C h\norm{\nabla v}_{L^2(\Omega)}\qquad\forall\ v\in H^1_{0}(\Omega).
	\end{equation}
	We define $d(x):=\operatorname{dist}(x,\Gamma)$ $\forall\ x\in\Omega\setminus\overline\Omega_1$.
	By assumption $d(x)\le C h$, hence $1\le \frac{C h^2}{\abs{d(x)}^2}$.
	\begin{equation*}
	\begin{aligned}
	\int_{\Omega\setminus\overline\Omega_1}\abs{v}^2\le C h^2\int_{\Omega\setminus\overline\Omega_1}\frac{\abs{v}^2}{\abs{d}^2}\le C h^2 \int_{\Omega}\frac{\abs{v}^2}{\abs{d}^2} \le C h^2 \int_{\Omega}\abs{\nabla v}^2,
	\end{aligned}
	\end{equation*}
	where we employed Hardy's inequality from Lemma~\ref{hardy_ineq}. Moreover:
	\begin{equation}\label{eq4}
	\norm{v}_{L^2(\Omega\setminus\overline\Omega_1)}\le \norm{v}_{L^2(\Omega)}\qquad\forall\ v\in H^1_{0}(\Omega).
	\end{equation}
	
	At this point, let us interpolate estimates \eqref{eq3} and \eqref{eq4}, getting
	\begin{equation}\label{eq5}
	\norm{v}_{L^2(\Omega\setminus\overline\Omega_1)}\le C h^s \norm{v}_{H^s_i(\Omega)}\qquad\forall\ v\in H^s_i(\Omega).
	\end{equation}
\end{proof}



\section*{Acknowledgements}
We would like to thank Pablo Antol\'in who provided us a tool to perform integration on trimmed geometries in GeoPDEs.

\bibliographystyle{siamplain}
\bibliography{bibliography.bib}
\end{document}